\documentclass[runningheads]{llncs}
\usepackage{mathrsfs}
\usepackage{amssymb}
\usepackage{amsmath}
\usepackage{mathtools}
\usepackage{graphicx}
\usepackage{xcolor}
\usepackage{lineno,hyperref}
\usepackage{multicol}
\definecolor{orange}{rgb}{1,0.5,0}

\definecolor{orcidgreen}{RGB}{166,206,57}
\newcommand{\orcidicon}{\includegraphics[width=0.26cm]{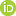}}
\def\orcidID#1{\renewcommand{\thefootnote}{\fnsymbol{footnote}}
  \unskip$^{\orcidicon}$\footnote{\orcidicon\color{orcidgreen}\,\footnotesize #1}\unskip
  \renewcommand{\thefootnote}{\arabic{footnote}}\unskip
}
\usepackage{lineno,hyperref}
\usepackage{amsmath}
\usepackage{amsfonts}
\usepackage{amssymb}
\usepackage{amsmath,amsfonts,amssymb}
\title{Weakly modular graphs with diamond condition, the interval function and  axiomatic characterizations }
\titlerunning{The axiomatic characterization}
\author{Lekshmi Kamal K. Sheela\inst{1}\orcidID{0000-0002-8527-3280}
  \and  Jeny Jacob\inst{1}\orcidID{0000-0001-6477-154X}  
  \and Manoj Changat\inst{1}\orcidID{0000-0001-7257-6031}}
\authorrunning{M.\ Changat, L.K. Sheela }
\institute{Department of Futures Studies, University of Kerala, Trivandrum
  695 581, India \email{lekshmisanthoshgr@gmail.com, jenyjacobktr@gmail.com, mchangat@keralauniversity.ac.in, }
  }
\date{ 26, April 2022}

\begin{document}

\maketitle
		\begin{abstract}
		 
		Weakly modular graphs are defined as the class of graphs that satisfy the \emph{triangle condition ($TC$)} and the \emph{quadrangle condition ($QC$)}. We study an interesting subclass of weakly modular graphs that satisfies a stronger version of the triangle condition, known as the \emph{triangle diamond condition ($TDC$)}. and term this subclass of weakly modular graphs as the \emph{diamond-weakly modular graphs}. It is observed that this class contains the class of bridged graphs and the class of weakly bridged graphs.
		 
		 The interval function $I_G$ of a connected graph $G$ with vertex set $V$ is an important concept in metric graph theory and is one of the prime example of a transit function; a  set function defined on the Cartesian product $V\times V$ to the power set of $V$ satisfying the expansive, symmetric and idempotent axioms.\\
		 In this paper, we derive an interesting axiom denoted as $(J0')$, obtained from a well-known axiom introduced by Marlow Sholander in 1952, denoted as $(J0)$. It is proved that the axiom $(J0')$ is a characterizing axiom  of the diamond-weakly modular graphs. We propose certain types of independent first-order betweenness axioms on an arbitrary transit function $R$ and prove that an arbitrary transit function becomes the interval function of a diamond-weakly modular graph if and only if $R$ satisfies these betweenness axioms. Similar characterizations are obtained for the interval function of bridged graphs and weakly bridged graphs. 
	\end{abstract}
	
\section{Introduction}\label{intro} 

In this paper, we consider only finite, simple, and connected graphs. \emph{Weakly modular graphs} form an interesting class of graphs and are introduced by Chepoi in \cite{Ch_metric}. Weakly modular graphs generalize a variety of  classes of graphs in \emph{metric graph theory} such as  modular graphs, Helly graphs, bridged graphs, and dual polar graphs.  Weakly modular graphs have applications beyond metric graph theory and studied in other fields of mathematics like geometric group theory, incidence geometries and buildings, theoretical computer science, and combinatorial optimization, see, Chalopine et al. in \cite{CCHO}, also, see the survey due to Bandelt and Chepoi \cite{BaCh_survey}. 

The main goal of the studies of graph classes in metric graph theory involves the investigation that how far the properties of the standard path metric $d$ in these graph classes can be approximated with the main properties of classical metric geometries like Euclidean $\ell_2$-geometry (and more generally, the $\ell_1$- and $\ell_{\infty}$-geometries), hyperbolic spaces, hypercubes, and trees. Several interesting classes of graphs are identified to have such properties. 

Weakly modular graphs and their subclasses are one of the central classes of metric graph theory. Bridged graphs form an interesting subclass of weakly modular graphs. A graph $G = (V, E)$  is a \emph{bridged graph} if $G$ has no isometric cycles of length greater than $3$. It is trivial to note that the family of bridged graphs contain the family of chordal graphs. In \cite{Ch_metric}, bridged graphs are characterized as the so called $(C_4$,  $ C_5)$-free weakly modular graphs and in \cite{faber}, Farber proved that a graph $G$ is a bridged graph if and only if all neighborhoods of convex sets are convex. Another interesting subclass of weakly modular graphs is the class  weakly bridged graphs, a superclass of bridged graphs, defined in \cite{CCHO}. A graph $G$ is called  weakly bridged if it is weakly modular without induced cycle of length four. %Another interesting subclass of weakly modular graphs is the class hereditary weakly modular graphs, a superclass of bridged graphs, introduced by Chepoi in \cite{Chepoi-89}. A graph $G$ is called \emph{hereditary weakly modular} if every isometric subgraph of $G$ is weakly modular.  
%Clearly the family of hereditary weakly modular graphs contain the family of bridged graphs.

In this paper, we introduce another subclass of weakly modular graphs, which we name as \emph{diamond-weakly modular graphs}. We prove that these graphs form a super-class of weakly bridged graphs and bridged graphs. We characterize these graphs using properties of the \emph{interval function} of the graph. The interval function $I_G$ of a graph $G$ can be related to \emph{metric betweenness} of $G$.

Recently, a new property in terms of the `natural betweenness', namely the metric betweenness in graphs is observed on the main classes of graphs in metric graph theory in \cite{Chalopin -2021}. That is, the central classes of graphs in metric graph theory  are capable of a characterization in first order logic (definable in first order logic) with the `` betweenness predicate" induced by the natural betweenness on graphs. More precisely,  the {\it metric betweenness} (or {\it shortest path betweenness}) resulting from the graph metric $d$ of graphs $G=(V,E)$ are defined using the ternary relation $B(abc)$ on the vertex set $V$ in such a way that ``the vertex $b$ lies on some shortest path of $G$ between the vertices $a$ and $c$.'' It is established in \cite{Chalopin -2021} that the first order logic with betweenness, denoted in short as \textbf{FOLB}, is a powerful logic as far as the first order axiomatization of graph properties are concerned. 

Metric betweenness can be expressed more conveniently by the  \emph{interval function}, defined for a connected graph $G$ as the function $I_G: V\times V \longrightarrow 2^{V}$ with   
\noindent
		\begin{align*}
	    I_G(u,v) & =\{w\in V: d(u,w) +d(w,v) = d(u,v) \} \\
	    &=\{w\in V: w \text{ lies on some shortest } u,v -\text{ path in } G \}.
	\end{align*} 

In this paper, we avoid the language of logic and use the interval function $I_G$ rather than the betweenness relation $B$ due to the simplicity of $I_G$.
The function $I_G$ is a well-known tool in metric graph theory and several authors have studied the function $I_G$, in particular, Mulder has given a systematic study of $I_G$ in the axiomatic setting in \cite{mu-80}. The interval function of a weakly modular graph has very special properties and a nice axiomatic characterization of $I$ exists for $G$ in terms of an arbitrary set function known as a transit function in \cite{ch-helly}. The term transit function is due to Mulder \cite{muld-08}, and it is introduced to generalize the three classical notions in mathematics, namely, convexity, interval and betweenness in an axiomatic approach. The concept of transit function is already known as interval operator and is used by Van de vel in \cite{VdV} in the context of convexity, see also, \cite{BaVdVVe,He}. Bandelt and Chepoi in \cite{ch-helly} discussed the same concept for studying metric properties of discrete metric spaces and graphs. In this paper, we follow Mulder and use the term transit function. 

Given a non-empty set $V$, a \emph{transit function} is defined as a function $R: V\times V \longrightarrow 2^{V}$ satisfying the following three axioms:
	
	\begin{description}
		\item [$(t1)$] $u \in R(u,v)$, for all $u,v \in V$,
		\item [$(t2)$] $R(u,v) = R(v,u)$, for all $u,v \in V$,
		\item [$(t3)$] $R(u,u) = \{u\}$, for all $u \in V$.
	\end{description}

We may refer $R$ as a transit function on $V$. If $V$ is the vertex set of a graph $G$, then we say that $R$ is a transit function on $G$. 

Given a transit function $R$ on $V$, one can define the underlying graph $G_{R}$ of a transit function $R$ on $V$ as the graph with vertex set $V$, where two distinct vertices $u$ and $v$ are joined by an edge if and only if $R(u,v) = \{u,v\}$. 

%The interval function denoted $I_G$ ( or $I$ if the graph $G$ is evident) of a connected graph $G$ is formally defined as the function $I_G: V\times V \longrightarrow 2^{V}$ defined with respect to the standard distance $d$ in $G$ as  
\noindent
	%	\begin{center}
%		$I_G(u,v)=\{w\in V$: $w$ lies on some shortest $u,v$ - path in $G \} =\{w\in V: d(u,w) +d(w,v) = d(u,v) \}$
%	\end{center}
%		\begin{align*}
%	    I_G(u,v) & =\{w\in V: d(u,w) +d(w,v) = d(u,v) \} \\
%	    &=\{w\in V: w \text{ lies on some shortest } u,v -\text{ path in } G \}
%	\end{align*} 
 
 Nebesk\'{y} initiated an interesting problem of characterizing  the interval function $I$ of a connected graph $G=(V, E)$ using a set of simple first-order axioms defined on an arbitrary transit function $R$ on $V$ during the 1990s. Nebesk\'{y} \cite{nebe-94,nebesky-94} proved that there exists such a characterization for the interval function $I(u,v)$ in terms of an arbitrary transit function $R$. More such characterizations are described in \cite{nebe-95,ne-08,nebesky-08,nebe-01,mune-09}.  %The axiomatic characterization of $I_G$ is extended to disconnected graphs in \cite{chfermuna-18}.  
 
In the rest of this section, we fix some of the graph theoretical notations and terminology used in this paper.  Let $G$ be a graph and $H$ a subgraph of $G$. The subgraph  $H$ is called an \textit{isometric} subgraph of $G$ if the distance $d_H(u,v)$ between any pair of vertices, $u,v$ in $H$  coincides with that of the distance $d_G(u,v)$.  $H$ is called an \textit{induced} subgraph if $u,v$ are vertices in $H$ such that $uv$ is an edge in $G$, then $uv$ must be an edge in $H$ also. %A path in $G$ which is induced as a subgraph is an \textit{induced path}. 
A graph $G$ is said to be $H$\emph{-free}, if $G$ has no induced subgraph isomorphic to $H$. A \emph{$k$-wheel} is a graph consisting of a cycle on $k$-vertices, $C_k$ ( $k\geq 4$), and a central vertex adjacent to all the vertices of the cycle. A wheel graph is a $k$-wheel for some $k\geq 4$.

We consider the following metric conditions on a graph $G$:
 \begin{itemize}
\item  Triangle condition $(TC)$: for any three vertices $u,v$, $w$ with $1 = d(v, w) < d(v, u) = d(u, w)$, there exists a common neighbor $z$ of $v$ and $w$ such that $d(u, z) = d(u, v) - 1$.\\
\item Quadrangle condition $(QC)$: for any four vertices $u,v$, $w$, $y$ with $d(v, y) = d(w, y) = 1$ and $2 = d(v, w) \leq d(u, v) = d(u, w) = d(u, y) - 1$, there exists a common neighbor $z$ of $v$ and $w$ such that $d(u, z) = d(u, v) - 1$.\\
\end{itemize}
It can be easily verified that the conditions $(TC)$ and $(QC)$ can also be defined in terms of the interval $I_G(u,v)$ instead of the distance $d$ in $G$. Also, it follows that graphs satisfying $(QC)$ are precisely the class of modular graphs \cite{CCHO}. 

A graph $G$ is \emph{weakly modular} if its distance function $d$ satisfies the triangle and quadrangle conditions. \\
%We introduce a stronger version of the triangle condition known as \emph{Triangle-Diamond condition $(TDC)$}, defined as follows.\\
%\item  Triangle-Diamond condition $(TDC)$: for any two vertices $v, w$ with $1 = d(v, w) < d(u, v) = d(u, w)$, there exists a common neighbor $z$ of $v$ and $w$ such that $d(u, z) = d(u, v) - 1$ and the $v,w,z$ should form diamonds with a vertex $x$ on $I(u,v)$ adjacent to $v$ and  a vertex $y$ on $I(u,w)$ adjacent to $w$. 

% and in \cite{faber}, Farber presented that a bridged graph if all neighborhoods of convex sets are convex. 

We organize the paper as follows. In Section~\ref{diawm}, we define diamond-weakly modular graphs and prove that the family of diamond-weakly modular graphs contain the class of bridged graphs and  weakly bridged graphs. Further, we prove that the family of diamond-weakly modular graphs are closed under the operation of gated amalgamations. Section~\ref{relations}, we provide the implications of the axioms that we consider in this paper for a general transit function $R$.  In Section~\ref{interval_function}, we characterize the diamond-weakly modular graph $G$ using the axiom $(J0')$ on its interval function $I_G$ and discuss the axiomatic characterization of the interval function of the class of diamond-weakly modular graphs using an arbitrary transit function $R$.  In Section~\ref{Bridged}, we provide the axiomatic characterizations of bridged graphs and weakly bridged graphs, respectively, using a set of first-order axioms on an arbitrary transit function.

\section{Diamond-weakly modular graphs}\label{diawm}
We first define the diamond-weakly modular graph, for that we introduce a stronger version of triangle condition known as  triangle diamond condition $(TDC)$, defined as follows.
\begin{itemize}
\item  Triangle diamond condition $(TDC)$: For any three vertices $u,v, w$ with $1 = d(v, w) < d(u, v) = d(u, w)$, there exists a common neighbor $z$ of $v$ and $w$ such that $d(u, z) = d(u, v) - 1$ and the vertices $v,w,z$ should form diamonds with  vertices $x$ and $y$ where $d(x,v)=d(y,w)=1$ and $d(u,x)= d(u,v)-1$, and $d(u,y)= d(u,w)-1$. %on $I(u,v)$ adjacent to $v$ and  a vertex $y$ on $I(u,w)$ adjacent to $w$. 
\end{itemize}
A graph $G$ is called \emph{ diamond-weakly modular}  if its distance function $d$ satisfies $(QC)$ and $(TDC)$.
%is a weakly modular and satisfies the Triangle-Diamond condition $(TDC)$.
It can be easily verified that the wheel graphs satisfy the $(TDC)$ and hence belong to the class of diamond-weakly modular graphs. Furthermore, diamond-weakly modular graphs contain the class of modular graphs by definition as it satisfies the quadrangle condition.
%A graph $G = (V, E)$  is a \emph{bridged graph} if $G$ has no isometric cycles of length greater than $3$. A graph $G$ is called \emph{hereditary weakly modular} if every isometric subgraph of $G$ is weakly modular. Hereditary weakly modular graphs are introduced by Chepoi in \cite{Chepoi-89}.  Clearly the family of hereditary weakly modular graphs contain the family of bridged graphs.\\ 

Now, we prove  the interesting subclasses of weakly modular graphs  such as bridged graphs and weakly bridged graphs are diamond-weakly modular. For that, we  define some more graph theoretic terms defined in \cite{fab}. 
A cycle $C$ is \emph{well-bridged} if and only if, for each $v$ in $C$, either the neighbors of $v$ in $C$ are  adjacent, or $d_G(v,x)< d_C(v,x)$ for some antipode $x$ of $v$ in $C$. An \emph{antipode} of a node $v$ of $C$ is a node of $C$ at maximum distance from $v$ along $C$. 

\begin{theorem}\cite{fab}\label{well-bridged}
   If $G$ is a bridged graph, then every cycle in $G$ is well-bridged.
  \end{theorem}
  Using this result, we have the following Lemma that bridged graphs are diamond-weakly modular graphs.
 \begin{lemma}\label{BrDWM}
  Bridged graphs satisfy $(TDC)$ and hence are diamond-weakly modular graphs.
\end{lemma}
  \begin{proof}
  Suppose $G$ is a bridged graph. Since  bridged graphs are weakly modular $G$  satisfies $(TC)$. Let $u$, $y$, $v$ be the vertices such that $1 = d(v,y) < d(u,v) = d(u,y)=m$ and let $z$ be the vertex adjacent to both $v$ and $y$ with $d(u,z) = m-1$. Suppose $x$ be the vertex adjacent to $y$ on $I(u,y)$ and $x'$ be the vertex  adjacent to $v$ on $I(u,v)$. Consider the cycle $C$: $xyz\rightarrow u \rightarrow x$ and has length at least four. Since $G$ is bridged, the cycle $C$ is well bridged by Theorem~\ref{well-bridged}. Consider the vertex $y$, then  either the neighbors of $y$ in $C$ are  adjacent, or $d_G(y,a)< d_C(y,a)$ for some antipode $a$ of $y$ in $C$. But we have antipode of $y$ in $C$ is $u$ and by assumption  $d_G(u,y)= d_C(u,y) =m$.  Then the only possibility is that the neighbors of $y$ in $C$ are  adjacent.  That is, $x$ and $z$ are  adjacent. Similarly, we can prove that $x'$ and $z$ are  adjacent. Hence $G$ satisfies $(TDC)$ and hence $G$ is a diamond-weakly modular graph. 
 \end{proof}
 
%In \cite{Ch_metric}, bridged graphs are characterized as the  $C_4$ and $C_5$-free weakly modular graphs.
Clearly, the class of bridged graphs form a strict subclass of diamond-weakly modular graphs. For example, $W_4$ and $W_5$ are diamond-weakly modular, but they are not bridged graphs.

 Weakly bridged graphs are superclass of bridged graphs. In \cite{CCHO}, weakly bridged graphs are defined as weakly modular graphs without an induced cycle of length four. Next Lemma shows that weakly bridged graphs are diamond-weakly modular graphs.
\begin{lemma}\label{wbrid}
 Weakly bridged graphs are diamond-weakly modular graphs.
\end{lemma}
\begin{proof}
Suppose $G$ is a weakly bridged graph which is not a diamond-weakly modular graph. Since $G$ is weakly modular, it satisfies $(TC)$, but it does not satisfy $(TDC)$. Let $u$, $y$, $v$ be the vertices such that $1 = d(v,y) < d(u,v) = d(u,y)=m$ and let $z$ be the vertex adjacent to both $v$ and $y$ with $d(u,z) = m-1$. Suppose $x$ be the vertex adjacent to $y$ on $I(u,y)$ and $x'$ be the vertex adjacent to $v$ on $I(u,v)$.  We may suppose that $z$ is not adjacent to $x$. Consider the cycle $C$: $xyz\rightarrow u \rightarrow x$, which has a length of at least four. Now consider the vertices $x$, $z$ and $u$, $2 = d(x, z)$ and $ d(u,x) = d(u,z)=m-1$. Apply $(QC)$ on $x$, $z$ and $u$, we  have a $w$ with $w$  adjacent to both $x$ and $z$ with $d(u,w) = m-2$. (Note that $w$ may be equal to $u$ if  cycle $C$ has length four). Then the vertices $y,x,z,w$ will induce a cycle of length four, a contradiction to the assumption that $G$ is weakly bridged.
\end{proof}

So far, we have  modular graphs, bridged graphs,  weakly bridged graphs, the graphs $W_4$ and $W_5$ are diamond-weakly modular graphs. Note that there are other diamond-weakly modular graphs which do not belong to any of these graphs and some of them are given in figures~\ref{dwm}.

\begin{figure}
\centering
\includegraphics[height=30mm]{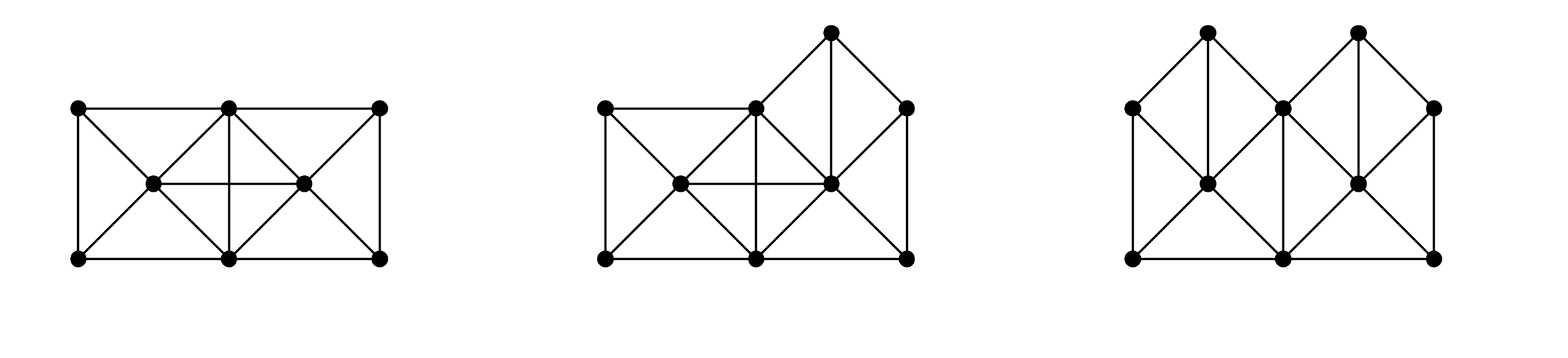}
\caption{diamond-weakly modular graphs}
\label{dwm}
\end{figure}

Besides these graphs, the graphs  obtained by identifying an  outer edge of $W_5$ with an edge in $K_n$ are diamond-weakly modular graphs. The graphs  obtained by identifying an outer edge of $W_4$ with an edge $uv$  in $K_n$ and making the central  vertex of $W_4$ adjacent to all the vertices in $K_n $ other than the vertices $u$ and $v$ are also diamond-weakly modular. Clearly, the graphs formed from any  vertex  identification of any of the different combination of $W_4$ and $W_5$ with  $W_4$, $W_5$ and bridged graph are diamond-weakly modular. 
    
 Let $G_1$ be $W_4$ or $W_5$ and $G_2$ be $W_4$, $W_5$ or bridged graph. Let $G_0$ be any isomorphic subgraph of both $G_1 $ and $G_2$. Now consider the graph $G$ obtained by identifying  $G_0$ of $G_1$ and $G_2$ and making all the vertices in $V(G_1)\setminus V(G_2)$  adjacent to all the vertices in $V(G_2)\setminus V(G_1)$. Then $G$ is diamond-weakly modular. \\
 
  Let $y \in V(G)$. Then a vertex $x$ in the set $S \subset V$ is a \emph{gate} in $S$ for $y$, if $x \in I(y,w)$ for each $w \in S$. The set $S$ is said to be \emph{gated} if every vertex in $G$ has a gate in $S$. It is clear that for a vertex $s\in S$, $s$ itself is the gate of $S$. That is, $S$ is gated if and only if every vertex $y\notin S$ has a gate in $S$. From the definition of a gated set $S$, it follows easily that the intersection of two gated sets is gated. The subgraph induced by a gated set $S$ is usually called a gated subgraph of $G$. \emph{Gated amalgamation} is an operation that can be performed on gated subsets (subgraphs) to form new gated sets (subgraphs). Gated amalgamation is well known operation in metric graph theory \cite{bandelt}.  It is defined for two gated sets $G_1$ and $G_2$ of $G$ as follows. 
%Consider the following definition of gated amalgamation by Bandelt et al.\cite{bandelt}.  
Let $G_1$ and $G_2$ be gated subgraphs of a graph $G$ such that $G_1\cup G_2$ = $G$ and $G_1\cap G_2\neq \emptyset$. If in addition there are no edges between $G_1 \setminus G_2$ and $G_2 \setminus G_1$ then $G$ is a gated amalgam of $G_1$ and $G_2$. Equivalently, let $H_1$ be a gated subgraph of a graph $G_1$ and $H_2$ a gated subgraph of $G_2$, where $H_1$ and $H_2$ are isomorphic graphs. Then the gated amalgam of $G_1$ and $G_2$ is obtained from $G_1$ and $G_2$ by identifying their subgraphs $H_1$ and $H_2$. 
%It is easy to see that these two description are indeed equivalent, see [1, Lemma 1]. We will call G1 and G2 covers of the gated amalgam G.
 \begin{theorem}
 Diamond-weakly modular graphs are closed under gated amalgamations.
\end{theorem}
\begin{proof}
Let $G_1$ and $G_2$ be two diamond-weakly modular graphs and $G$ be the graph obtained by the gated amalgamation of $G_1$ and $G_2$. Let  $H_1\sim H_2 =G_0$ be the isomorphic gated subgraphs of $G_1$ and $G_2$ which are identified to obtain $G$. Now consider the vertex $u$ and the edge $vy$ in $G$ with  $d(u,v)=d(u,y)=m$. Then the edge $vy$ lies either  in $G_1 $ or in $G_2 $, since there are no edges between $G_1 \setminus G_2$ and $G_2 \setminus G_1$. Clearly,  these three vertices satisfies $(TDC)$ if they all lie either in $G_1$ or in   $G_2$. Suppose the vertex $u$ lies in $G_1 \setminus G_2$ and the edge $vy$ lies in $G_2 \setminus G_1$. Let $g(u)$ be the gate of $u$ in $G_2$. Then $g(u)$ will be the unique vertex in $G_2$ closest to $u$ and $g(u)\in I(u,v) \cap I(u,y)$. Since $d(u,v)=d(u,y)=m$ and  $g(u)$ is the gate of $u$ in $G_2$, we get, $d(g(u),v) = d(g(u),y)$. Then the vertices $v$, $y$, $g(u)$, with $d(v,y)=1$ and $d(g(u),v) = d(g(u),y)$ satisfies $(TDC)$ on $G$, since, they lie in the diamond-weakly modular graph $G_2$. Hence the vertices $u$, $v$, $y$ satisfy $(TDC)$ on $G$. Similarly, we can prove that $G$ satisfies $(QC)$.   %there exists a common neighbor $z$ of $v$ and $w$ such that $d(g(u), z) = d(g(u), v) - 1$. Moreover,  $d(u,z) = d(u,v) - 1$ and $z$ is adjacent to . Hence $G$ satisfies (TC). Similarly, we can prove that, $G$ satisfies (QC). 
\end{proof}

Note that, diamond-weakly modular graphs are not closed under Cartesian products. For example, consider the graph $K_3 \Box K_2$. It follows that the graph $K_3\Box K_2$ (the prism graph) is a weakly modular graph that doesn't satisfy $TDC$.\\ We also have  the following proposition.

\begin{proposition}\label{gdwm}
    A graph $G$ is diamond-weakly modular then $G$ is weakly modular and if $G$  contain $ C_5,W_4^-$ or house as induced subgraphs then there exists a vertex adjacent to all the five vertices of these subgraphs. 
\end{proposition}

%\begin{proof}
%\begin{figure}
%\centering
%\includegraphics[height=30mm]{dwm-free.png}
%\caption{Forbidden subgraph of diamond-weakly modular}
%\label{dwm-free}
%\end{figure}
\begin{proof}
Suppose  $G$ is a diamond-weakly modular graph. Then clearly, $G$ is a weakly modular graph. Assume $G$ contains $ C_5$  as an induced subgraph with $u,x,y,v,x'$ as its adjacent vertices. For the vertex $u$ and the edge $yv$ with $d(u,y)=d(u,v)=2$, there is a vertex $z$ adjacent to all the $y,v, x, x'$ and $u$, since $G$ is diamond-weakly modular.  So there exists a vertex $z\in G$ such that $z$ is adjacent to all the vertices of $C_5$. Similar situation hold in the case of induced $ W_4^-$ and house.
\end{proof}
The converse of above Proposition~\ref{gdwm}  need not  be true and is illustrated in Figure~ \ref{notdwm}
\begin{figure}%[]
\begin{tabular}{ccc}
 \begin{minipage}{0.30\textwidth}
 \includegraphics[width=\textwidth]{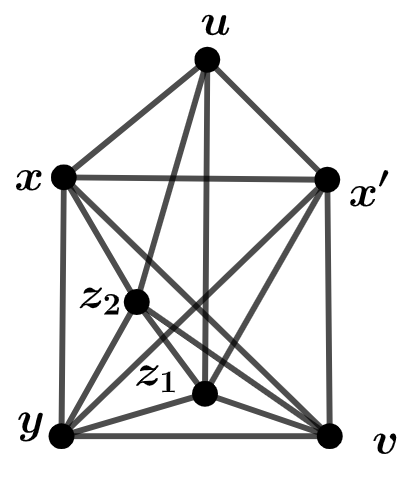}\label{notdwm}
 \end{minipage} & \quad\qquad
 &\begin{minipage}{0.60\textwidth}
 \caption{ The graph in the figure is weakly modular and for the induced subgraphs, $C_5$, house, and $W_4^-$, there exists a vertex adjacent to all the five vertices of these subgraphs. But it is not diamond-weakly modular, since the vertex $u$ and the edge $yv$ satisfies $(TC)$ but not $(TDC)$. Here $z_1$ is not adjacent to $x$ and $z_2$ is not adjacent to $x'$. Also, the vertices $u,x,y,z_1$ and $v$ induce a house and $z_2$ is a vertex adjacent to all the five vertices of the house. But still, the graph do not satisfy $(TDC)$.}
 \end{minipage}
\end{tabular}
\end{figure}

\section{ Betweenness axioms for arbitrary transit function and their implications} %of diamond-weakly modular graphs and relationship of these axioms of an arbitrary transit function $R$}\label{relations}

In this section, we consider a transit function $R$ defined on a non-empty finite set $V$ and discuss the axioms and their relationships for $R$. These axioms will be used for the characterization of diamond-weakly modular graphs using the interval function as well as the axiomatic characterization of the interval function of diamond-weakly modular graphs, bridged graphs and weakly bridged graphs using an arbitrary transit function $R$. \\

The following three axioms denoted as $(b2)$, $(b3)$, and $(b4)$ together with the defining transit axioms $(t1)$, $(t2)$  are  essential and considered by Nebesk\'{y} and Mulder in the characterizations of the interval function $I$ of a graph. \\ 
	
	\begin{tabular}{lclclc} 
		& $(b2)$ &if $x\in R(u,v)$ and $y\in R(u,x)$, then $y\in R(u,v)$, & \\ 
		& $(b3)$ &if $x\in R(u,v)$ and $y\in R(u,x)$, then $x\in R(y,v)$, & \\
		& $(b4)$ &if $x\in R(u,v)$, then $R(u,x)\cap R(x,v)= \{x\}$. & 
	\end{tabular}\\
	
\noindent A transit function $R$ satisfying axioms $(b2)$ and $(b3)$ is known as a \emph{geometric transit function}. %\cite{BaCh_helly, VdV}. %The same line of research was attempted by Bandelt in \cite{} and obtained that the geometric axioms together with e axioms  
Bandelt and Chepoi et al. in \cite{ch-helly} introduced the axiom known as $(ta)$ for proving a necessary condition for a  geometric transit function to be the interval function of a connected graph. Further in \cite{ch-helly}, it is  proved that the axiom $(ta)$ is satisfied by the interval function of a weakly modular graph. 

\begin{description}
\item[$(ta)$:] if  $ R( u,v) \cap R( u,w) =\{  u \} , R( u,v) \cap R( v,w) =\{  v \}, R( u,w) \cap R( v,w) =\{  w \} $ and $ R( u,v) = \{ u,v  \}  $, then $  R( u,w) = \{ u,w  \}   $ and $  R( v, w) = \{ v, w  \}  $, for all  $ u,v,w \in V $.
\end{description}
The following axiom $(J{0})$ is due to Sholander in his attempt to prove the interval function of trees in \cite{Sholander-52} during 1952. In \cite{Changat-22}, this axiom is shown to be a characterizing axiom for the interval function of a Ptolemaic graph, see also \cite{ch-lek-pr-2022}. 

\begin{description}
\item[$(J{0})$:] if  $x\in R(u,y), y\in R(x,v)  $, then $ x\in R(u,v)$, for all distinct $u,v,x,y\in V$.
\end{description} 
We consider a slightly modified version of the axiom $(J0)$, defined  as follows;

 \begin{description}
\item[$(J0')$:] if $x\in R(u,y), y\in R(x,v), R(u,y)\cap R(x,v) \subset \{u,x,y,v\}$, then $ x\in R(u,v)$,  for all distinct  $u,x,y,v \in V$.
\end{description}
We prove that this axiom is a characterizing axiom of the  diamond-weakly modular graphs in Section~\ref{interval_function}, where %This special family of weakly modular graphs generalize two interesting subclasses of weakly modular graphs, namely the modular graphs and bridged graphs.  Also this subclass of weakly modular graphs contain the family of hereditary weakly modular graphs. 
we give a characterization of  diamond-weakly modular graphs $G$ using its interval function $I_G$ and also provide an axiomatic characterization of $I_G$ in terms of a set of first order axioms on an arbitrary transit function $R$. %We also prove a first order axiomatic characterization of the class of bridged graphs and hereditary weakly modular graphs. 

From the definitions of the axioms $(J0)$ and $(J0')$, we have that the axiom $(J0)$ implies axiom  $(J0')$, while the reverse implication is not true.  In other words, axiom  $(J0')$ is a weaker axiom than $(J0)$.  Example~\ref{j0'nj0} shows that  $(J0')\not\Rightarrow (J0)$.
\begin{example} For a transit function $R$ on $V$, {$(J0')\not\Rightarrow (J0)$}\label{j0'nj0}$~$\\
	%From the definitions, it follows that  $(J0)\implies (J0')$. For proving that  $(J0')\centernot \implies (J0)$, we have the following example. 
	Let $V=\{a,b,c,d,e\}$.  Let $R:V\times V\rightarrow 2^V$ defined as follows. $R(a,e)=\{a,e\}, R(a,b)=\{a,b\}, R(b,e)=\{b,e\}, R(b,c)=\{b,c\}, R(c,e)=\{c,e\},R(c,d)=\{c,d\}, R(d,e)=\{d,e\},$ $R(a,c)=\{a,b,c,e\}$, $R(a,d)=\{a,e,d\},$ $R(b,d)=\{b,c,d,e\}$, $R(x,y)=R(y,x)$ and $R(x,x)=\{x\} $, for all $x,y$ in $V$.  We can see that $b\in R(a,c)$ and  $c\in R(b,d)$ but $b\notin R(a,d)$, so that $R$ does not satisfy $(J0)$ axiom.  We can see that there exists no $u,v,x,y$ and $z$ satisfying the assumptions of the axiom $(J0')$ and hence the axiom $(J0')$ follows trivially. 
	
\end{example}

In addition to the axiom  $(J0')$, we consider the axioms  $(br)$ and $(br')$ for a transit function $R$. These axioms are essential for  the characterizations of bridged graphs and weakly bridged graphs, which we prove in Section~\ref{Bridged}. %Let $V$ be a nonempty set and $R$ be a transit function on $V$, 
	\begin{description}
		%\item[$(J{0})$]: For any distinct vertices $u,v,x,y\in V$ we have $x\in R(u,y), y\in R(x,v)  \implies x\in R(u,v)$.
	%	\item[$(J0')$]   $x\in R(u,y), y\in R(x,v), R(u,y)\cap R(x,v) \subset \{u,x,y,v\}\implies x\in R(u,v)$, \,\,\, for all distinct  $u,x,y,v \in V$.
		%	\item[$(J1)$ ]: $w\in R(u,v),w\not = u,v, \implies $ there exists $x\in R(u,w)\setminus R(v,w),y\in R(v,w)\setminus R(u,w)$, such that $R(x,w)=\{x,w\},R(y,w)=\{y,w\}$ and $w\in R(x,y)$
		%\item[$(J2)$] $ R(u,x) = \{u,x\}, R(x,v) = \{x,v\}, R(u,v) \neq  \{u,v\} \implies  x \in  R(u,v)$,  for all distinct  $u,x,y,v \in V$.
		%\item[$(J{2'})$:] $x\in R(u,y), y\in R(x,v),  R(u,x) =\{u,x\}, R(x,y)=\{x,y\}, R(y,v) = \{y,v\}, R(u,v)\neq \{u,v\}\Rightarrow x\in R(u,v)$
		%\item[$(J{3})$] :  $x\in R(u,y), y\in R(x,v), x\neq y, R(u,v)\neq \{u,v\}\Rightarrow x\in R(u,v)$. \\ 
		%\item[$(J{3'})$]:  $x\in R(u,y), y\in R(x,v), R(x,y)\neq \{x,y\}, R(u,v)\neq \{u,v\}\Rightarrow x\in R(u,v)$.
		 %\item[$(tc)$:] any three vertices $ u,v,w $ if $ R( u,v) \cap R( u,w) =\{  u \} , R( u,v) \cap R( v,w) =\{  v \}, R( u,w) \cap R( v,w) =\{  w \} $ and $ R( u,v) = \{ u,v  \}  \Rightarrow  R( u,w) = \{ u,w  \}   $ and $  R( v, w) = \{ v, w  \}  $.
		 \item[$(br)$:] if $R(x,y)=\{x,y\}, R(x,u)=\{x,u\}, R(v,y)=\{v,y\} $ and $z\in R(u,v)$, then $ R(x,z)=\{x,z\}$ or $R(y,z)=\{y,z\}$, for any $u,v,x,y,z$.
		  \item[$(br')$:] if $ R(u,x)=\{u,x\}, R(x,v)=\{x,v\} $ and $z\in R(u,v)$, then $ R(x,z)=\{x,z\}$, for any $u,v,x,z$.
		
	%	 \item[$(hr)$:]  $R(x,y)=\{x,y\}, R(x,u)=\{x,u\}, R(v,y)=\{v,y\} $, $x\notin R(u,v)$ and $z\in R(u,v)\implies R(x,z)=\{x,z\}$ or $R(y,z)=\{y,z\}$, for any $u,v,x,y,z$.
	\end{description}
%From the definition of the axioms, we observe the following. The axiom $(J2)$ is a simple betweenness axiom which is always satisfied by the interval function $I$ and the induced path transit function $J$ of a graph $G$. 
We now state the Mulder- Nebesk\'{y} Theorem stated in  \cite{mune-09}
characterizing the interval function of an arbitrary connected graph using axioms on an arbitrary transit function. In addition to the transit axioms $(t1), (t2)$ and the betweenness axioms $(b2), (b3), (b4)$, the following axioms $(s1)$ and $(s2)$ are required in the theorem.

\begin{description}
\item[$(s1)$]:   $ R( u,\bar{u}) = \{ u,\bar{u}  \}, R( v,\bar{v}) = \{ v,\bar{v}\}, u\in R( \bar{u}, \bar{v} ) $ and $ \bar{u} , \bar{v} \in R( u,v) $, then  $ v\in R( \bar{u}, \bar{v})  $.
\item[$(s2)$]:  $ R( u,\bar{u}) = \{ u,\bar{u}  \}, R( v,\bar{v}) = \{ v,\bar{v}  \},   \bar{u}  \in R( u,v) $, $ v\notin R( \bar{u}, \bar{v} ) ,  \bar{v}  \notin R( u,v)  $,   then $ \bar{u}  \in R( u,\bar{v} ) $.
\end{description}

\begin{theorem}\label{interval}\cite{mune-09}
  Let $R:V\times V \longrightarrow 2^{V}$ be a function on $V$, satisfying the axioms $(t1), (t2), (b2), (b3),(b4)$  with the underlying graph $G_R$ and let $I$ be the interval function of $G_R$. The following statements are equivalent.
  \begin{itemize}
      \item[$(a)$] $R=I$
      \item[$(b)$] $R$ satisfies axioms $(s1)$ and $(s2)$.
  \end{itemize}
\end{theorem}

The axiom $(ta)$ defined by Chepoi et al. in \cite{ch-helly} is needed for the characterization of $I_G$ of these graph classes (diamond-weakly modular,  bridged graphs and weakly bridged graphs). Also, it is proved in \cite{ch-helly} that  discrete geometric interval space $X$ satisfying the $(ta)$ axiom is graphic. In our terminology, it means that if $R$ satisfies axioms $(b2)$, $(b3)$ and $(ta)$ then $R$ will coincide with the interval function of the underlying graph $G_R$.
%We now state the Mulder- Nebesk\'{y} Theorem in  \cite{mune-09} characterizing the interval function of an arbitrary connected graph using axioms on an arbitrary transit function.  By Mulder-Nebesky theorem, we have $R=I_{G_R}$ if and only if $R$ satisfies axioms $(s1)$ and $(s2)$, where $R$ is a transit function satisfying axioms $(b2)$ and $(b3)$. So combining these two results we get the fact that the axioms $(b2),(b3)$ and $(ITC)$ will imply the axioms $(s1)$ and $(s2)$.
%In addition to the transit axioms $(t1), (t2)$ and the betweenness axioms $(b2), (b4)$, the following axioms $(s1)$ and $(s2)$ are required in the Theorem. Using the Mulder- Nebesk\'{y} Theorem and the Theorem~\ref{bridged-J0'}, we prove the Theorem characterizing the interval function of a bridged graph.

Though, we give another proof establishing the dependency of the axioms $(b2)$, $(b3)$,  $(ta)$, $(s1)$ and $(s2)$.

% Also we have, if a transit function $R$ satisfies $ b_{2} $ and $ b_{3} $ then it satisfy $b_{1} : x\in R( u,v) \Rightarrow v\notin  R( u,x) $ .
\begin{lemma}\label{s1s2}
Let $ R $ be a transit function on a nonempty set $V$  satisfies the axioms $ (b{2}) $ and $ (b{3}) $ then for any $u,v$ and $w$, there exists   $ x\in R( u,v) \cap R( u,w) $ such that  $ R( x,v) \cap R( x,w) = \{x\}$.
\end{lemma}
\begin{proof}

We have $ u\in R( u,v) \cap R( u,w) $.  If $  R( u,v) \cap R( u,w) = \{u\} $, then $ u $ is the required $ x $ and completes the proof. If not, there exist a $ x_{1} $ such that $ x_{1} \in R( u,v) \cap R( u,w) $. Then $ x_{1} \in R( u,v) $ and $ x_{1}  \in R( u,w) $  and by $ (b1) $ we have $ u\notin R( x_{1},v) $ and $ u\notin R( x_{1},w) $ imply that $ R( x_{1},v) \subset R( u,v)  $ and $ R( x_{1},w) \subset R( u,w)  $ by $ (b2) $. Hence 
 $   R( x_{1},v) \cap  R( x_{1},w)   \subset   R( u,v) \cap R( u,w) $ and 
 $  \mid  R( x_{1},v) \cap  R( x_{1},w) \mid <  \mid   R( u,v) \cap R( u,w) \mid  $. If $  R( x_{1},v) \cap  R( x_{1},w) = \{x_1\} $, complete the proof with $ x = x_{1} $. Otherwise, there exist $ x_{2} $ such that $ x_{2} \in R( x_{1},v) \cap R( x_{1},w) $ and by the above argument we have $ R( x_{2},v) \cap  R( x_{2},w)   \subset   R( x_{1},v) \cap R( x_{1},w)  \subset   R( u,v) \cap R( u,w)  $ and 
 $ \mid   R( x_{2},v) \cap  R( x_{2},w) \mid  <  \mid   R( x_{1},v)   \cap R( x_{1},w) \mid   < \mid   R( u,v) \cap R( u,w) \mid $.  If $  R( x_{2},v) \cap  R( x_{2},w) = \{x_2\}$, complete the proof with $ x = x_{2} $.
 Otherwise, since $ V $ is finite, we can find $ x_{n} $ with  $ x_{n} \in R( x_{n-1},v) \cap R( x_{n-1},w) $ such that   $  R( x_{n},v) \cap  R( x_{n},w)   \subset   R( x_{n-1},v) \cap R( x_{n-1},w)    \subset... \subset   R( u,v) \cap R( u,w)    $ and 
 $ \mid   R( x_{n},v) \cap  R( x_{n},w)\mid 
   < \mid   R( x_{n-1},v)   \cap R( x_{n-1},w) \mid 
    <... < \mid   R( u,v) \cap R( u,w) \mid $ and
 $   R( x_{n},v) \cap  R( x_{n},w) = \{x_n\}  $. 
    \end{proof}	  
    \begin{proposition}\label{Ws1}
	A transit function $ R $ defined on a non-empty set $ V $ satisfies axioms $ (b2) ,(b3)  $	and $ (ta) $ then $ R $ satisfies axioms $ (s1) $ and $(s2)$.
\end{proposition}
\begin{proof}   
	Let $ R $ be a transit function satisfying $ (b2)$, $ (b3) $ and $ (ta)$. 
	%:  $ any three vertices $ u,v,w $ if $ R( u,v) \cap R( u,w) =\{  u \} , R( u,v) \cap R( v,w) =\{  v \}, R( u,w) \cap R( v,w) =\{  w \} $ and $ R( u,v) = \{ u,v  \}  \Rightarrow  R( u,w) = \{ u,w  \}   $ and $  R( v, w) = \{ v, w  \}  $
We have to show that $ R $ satisfies $ (s1) $. Suppose not. That is $ R( u,\bar{u}) = \{ u,\bar{u}  \}, R( v,\bar{v}) = \{ v,\bar{v} \},  u\in R( \bar{u}, \bar{v} ) $ and $ \bar{u} , \bar{v} \in R( u,v) $, and $ v\notin R( \bar{u}, \bar{v} )$. Consider the vertices $  \bar{u}, v,  \bar{v} $.\\
\textbf{claim 1:} $ \bar{v} \notin R( \bar{u}, v ) $ and  $ u\notin R( \bar{u}, v )$.
	
	If  $ \bar{v} \in R( \bar{u}, v ) $, we have $ \bar{u} \in R( u, v ) $ and $ \bar{v} \in R( \bar{u}, v )  \Rightarrow \bar{u} \in R( u, \bar{v} ) $ by $ ( b{3}) $ and then $ u\in R( \bar{u}, \bar{v} ), \bar{u} \in R( u, \bar{v} ) \Rightarrow u \in R( \bar{u}, \bar{u} )
	$, a contradiction. Therefore $ \bar{v} \notin R( \bar{u}, v ) $. Also since $ \bar{u}\in R( u, v )$, by $(b1)$  we have $ u\notin R( \bar{u}, v )  $. Hence the claim.\\
	\textbf{Claim 2:} $   R( \bar{u}, \bar{v} ) \cap  R( \bar{u}, v )  \setminus \{ \bar{u} \} \neq \phi $.
	
	If $R( \bar{u}, \bar{v} ) \cap  R( \bar{u}, v )  \setminus \{ \bar{u} \} = \phi  $,  then $  R( \bar{u}, \bar{v} ) \cap  R( \bar{u}, v ) = \{\bar{u} \} ,  R( \bar{u}, \bar{v} ) \cap  R( \bar{v}, v ) = \{ \bar{v} \} $ and $  R( \bar{u}, v ) \cap  R( \bar{v}, v ) = \{ v \} $ and $ R( v,\bar{v}) = \{ v,\bar{v}  \} $ then by $ (ta) $ we can say that $ R( \bar{u}, \bar{v} ) =\{  \bar{u}, \bar{v} \}  $, a contradiction since $ u \in R( \bar{u}, \bar{v} )$. Therefore there exist a $ w_{1} $ with $ w_{1} \neq  \bar{u} $ such that $  w_{1} \in   R( \bar{u}, \bar{v} ) \cap  R( \bar{u}, v ) $. Hence the claim.
	Since  $  w_{1} \in   R( \bar{u}, \bar{v} ) \cap  R( \bar{u}, v ) $ by  lemma~\ref{s1s2}, we can find  $   y_{1} $ such that $  R(  y_{1} , \bar{v} ) \cap  R( y_{1}, v ) = \{  y_{1} \} $, and by $ (ta) $ on $  y_{1}, v, \bar{v} $, we have $  R(  y_{1} , \bar{v} ) =  \{ y_{1} , \bar{v} \}  $ and $  R(  y_{1} , v ) =  \{ y_{1} , v \}  $.
	Since   $ \bar{u} \in R( u, v ),  y_{1} \in   R( \bar{u}, v ) $ by $ (b{3}) $ we have  $ \bar{u} \in   R( u, y_{1} ) $. 
	Now consider the vertices $ u, \bar{v},y_{1}  $, then analogous to claim1, here we can prove  that $ \bar{v} \notin R( u, y_{1} ) $ and $ \bar{u} , y_{1} \notin R( u, \bar{v} )  $ by  $ (b{3}) $.
	Also we can prove $R(u, \bar{v} ) \cap  R( u, y_{1} )    \setminus \{ u \} \neq \phi $
	by applying $ (ta) $ on $ u, \bar{v},y_{1}  $, therefore there exist a $ y_{2} $ with $ y_{2} \neq  u $ such that $  y_{2} \in   R(u, \bar{v} ) \cap  R( u, y_{1} )   $
	and   $  R(  y_{2} , \bar{v} ) \cap  R( y_{2}, y_{1} ) = \{  y_{2} \} $. Then by $ (ta) $ on $  y_{2}, y_{1}, \bar{v} $, we have $  R(  y_{2} , \bar{v} ) =  \{ y_{2} , \bar{v} \}  $ and $  R(  y_{2} , y_{1} ) =  \{ y_{2},  y_{1} \}  $.
	Since  $ u \in R( \bar{u}, \bar{v} ),  y_{2} \in   R( u, \bar{v} ) $, we have  $ u \in   R( \bar{u}, y_{2} ) $ by $ (b{3}) $. 
	Now consider $ \bar{u} , y_{1}, y_{2} $, and apply the same as above, we can find a $ y_{3} $  such that $  y_{3} \in   R(  \bar{u}, y_{1} ) \cap  R( \bar{u}, y_{2} )   $ and  $  R(  y_{3} , y_{1} ) =  \{ y_{3},  y_{1} \}  $ and  $  R(  y_{3} , y_{2} ) =  \{ y_{3},  y_{2} \}  $.
	Also  $ \bar{u} \in   R( u, y_{3} ) $. 
	Continuing like this we can find $ y_{4},  y_{5}, ... $ such that  $  R(  y_{k} , y_{k-1} ) =  \{ y_{k},  y_{k-1} \}  $ and  $  R(  y_{k} , y_{k-2} ) =  \{ y_{k},  y_{k-2} \}  $  and for some $ k=2m $,  $ u \in   R( \bar{u}, y_{k} ) $ and
	$ \bar{u} \in   R( u, y_{k-1} ) $.  
	Since $ V $ is finite, we can find some $ n $ such that either $ y_{n}=u $ or  $ y_{n}=\bar{u} $.
	If $ y_{n}=u $, we have $  R(  y_{n} , y_{n-1} ) =  \{ y_{n},  y_{n-1} \}  $ and  $  R(  y_{n} , y_{n-2} ) =  \{ y_{n},  y_{n-2} \}  $  and 
	$ u \in   R( \bar{u}, y_{n} ) $ and also $ \bar{u} \in   R( u, y_{n-1} ) $, a contradiction since
	 $ R( u, y_{n-1} )= R(  y_{n} , y_{n-1} ) =  \{ y_{n},  y_{n-1} \}$.
	If $ y_{n}= \bar{u} $, then  $  \bar{u} \in   R( u, y_{n} ) $ and also  $ u \in   R( \bar{u}, y_{n-1} ) $, a contradiction since  $R( \bar{u}, y_{n-1} )=  R(  y_{n} , y_{n-1} ) =  \{ y_{n},  y_{n-1} \}  $. In both case we have a contradiction.	Therefore  $ v\in R( \bar{u}, \bar{v} )$.
	
	Now we have to show that $R$ satisfies axiom $(s2)$. Suppose not. That is $ R( u,\bar{u}) = \{ u,\bar{u}  \}, R( v,\bar{v}) = \{ v,\bar{v}  \},   \bar{u}  \in R( u,v) $, $ v\notin R( \bar{u}, \bar{v} ) ,  \bar{v}  \notin R( u,v)  $ and $ \bar{u}  \notin R( u,\bar{v} ) $.\\
	\textbf{Claim 3:} $ u\notin R( \bar{u}, \bar{v} ) $.
	Suppose $ u\in R( \bar{u}, \bar{v} ) $. 
	Consider the vertices $ \bar{u}, \bar{v} ,v $.  Since $\bar{u} \in R(u,v)$ by $(b1)$ we have $ u\notin R( \bar{u}, v ) $. Also $ \bar{v} \notin R( \bar{u}, v ) $ since $\bar{v}  \notin R( u,v)  $ and $   v\notin R( \bar{u}, \bar{v} ) $. From these we have $ R( \bar{u}, v ) \cap R( \bar{v}, v ) = \{  v \}$,
	$ R( \bar{u}, \bar{v} ) \cap R( \bar{v}, v ) = \{  \bar{v} \}   $ and $  R( \bar{v}, v ) = \{  \bar{v} ,v \} $ also  $ u\in R( \bar{u}, \bar{v} ) $. Then by $(ta)$, $ R( \bar{u}, v ) \cap R( \bar{u}, \bar{v} )\setminus \{ \bar{u} \} \neq \phi$.  That is there exist $ w_{1} $ such that $ w_{1} \in R( \bar{u}, v ) \cap R( \bar{u}, \bar{v} ) $.
	Hence by  lemma~\ref{s1s2}, there exist a $ y_{1} $ such that $ R( y_{1}, v ) \cap R( y_{1}, \bar{v} ) = \{y_{1}\}$.  
	Then by $(ta)$ on $  y_{1}, \bar{v}, v $, we have $  R( y_{1}, \bar{v} ) =  \{ {y_{1}}, \bar{v} \}$ and $  R( y_{1}, v  ) =\{  {y_{1}, v} \}$.
	Also $ \bar{u} \in R( u, y_{1} ) $. That is we have $ R( u,\bar{u}) = \{ u,\bar{u}  \}, R( \bar{v}, y_{1} ) = \{ \bar{v}, y_{1}  \},  u,y_{1}\in R( \bar{u}, \bar{v} ),  $ and $  \bar{u} \in R( u, y_{1} ) $, then by $(s{1}) $  $  \bar{v} \in R( u, y_{1} ) \subseteq R( u,v ) $. That is $ \bar{v} \in  R( u,v ) $, a contradiction. Therefore $ u\notin R( \bar{u}, \bar{v} ) $ and hence claim 3.
	
	Now consider $ u,\bar{u}, \bar{v} $.
	We have  $ u\notin R( \bar{u}, \bar{v} ),  \bar{u}  \notin R( u,\bar{v} ) $  and $ v\notin R( \bar{u}, \bar{v} ) $. Also $ v\notin R( u, \bar{v} )$, since $ \bar{u}\notin R( u, \bar{v} )$. Apply $(ta)$ on $ u,\bar{u}, \bar{v} $, we have two cases\\
	\textbf{case 1} 
	$ R( u,\bar{v} ) \cap R( \bar{u}, \bar{v} ) = \{  \bar{v} \}$. Which implies that $ R( \bar{u}, \bar{v} ) = \{  \bar{u}, \bar{v} \}  $ and $ R( u, \bar{v} ) = \{  u, \bar{v} \}  $ by $(ta)$.
	Then consider $ v,\bar{u}, \bar{v} $, also it is clear that $ R( \bar{u}, \bar{v} ) = \{  \bar{u}, \bar{v} \}  $ and $ R( v, \bar{v} ) = \{ v, \bar{v} \} $, then by $(ta)$ on $ v,\bar{u}, \bar{v} $, we have 
	$ R( \bar{u}, v ) = \{  \bar{u}, v \}  $.
	Consider $u,v, \bar{v} $. We have $ \bar{u}, v \notin  R( u, \bar{v} )  $ and $    \bar{v} \notin R( u,v ) $. Then $ R( u, \bar{v} ) \cap  R( v, \bar{v} ) = \{ \bar{v} \} , R( u, v ) \cap  R( v, \bar{v} ) = \{ v \} $ and $ R( u, v ) \cap  R( u, \bar{v} ) = \{ u \} $ and $  R( v,\bar{v}) = \{ v,\bar{v}  \} $, by $(ta)$ we have $  R( u,v) = \{ u,v  \} $, contradiction since $ \bar{u} \in R( u,v)$.\\
\textbf{case 2}  
	$ R( u,\bar{v} ) \cap R( \bar{u}, \bar{v} ) \neq \{  \bar{v} \} $. That is there exist a $ x_{1} $ such that $ x_{1} \in  R( u,\bar{v} ) \cap R( \bar{u}, \bar{v} ) $, and by  lemma~\ref{s1s2},  we may assume that $ R( x_{1}, u ) \cap R( x_{1}, \bar{u} ) = \{  x_{1} \} $ and by $(ta)$ on $ u,x_{1},  \bar{u}  $, we have $ R( u, x_{1} ) = \{   u, x_{1} \}   $ and $ R( \bar{u}, x_{1} ) = \{   \bar{u}, x_{1} \} $.
	Then $ x_{1} \in  R( u,\bar{v} ),  $ but  $ x_{1} \notin  R( \bar{u},v )  $. Now	consider $ \bar{u}, x_{1},v $.\\
    \textbf{Claim 4:} $\bar{u} \notin R(  x_1, v ) $. Suppose $\bar{u} \in R(  x_1, v ) $. Consider the vertices $\bar{u}, \bar{v}, v$, we have $R(\bar{u},\bar{v}) \cap R(v,\bar{v})=\{\bar{v}\}$, $R(\bar{u},v) \cap R(v,\bar{v})=\{v\}$, $R(v,\bar{v})= \{v,\bar{v}\}$. If $R(\bar{u},\bar{v}) \cap R(\bar{u},v)=\{\bar{u}\}$, then by $(ta)$,we get $R(\bar{u},\bar{v})= \{\bar{u},\bar{v}\}$, a contradiction as   $ x_{1} \in  R(\bar{u},\bar{v} )$. So there exist a $x_2$ such that $ x_2 \in R(\bar{u},\bar{v}) \cap R(\bar{u},v)$.  $\bar{u}\in R(x_1,v)$ and $ x_2 \in R(\bar{u},\bar{v})$ implies $\bar{u} \in R(x_1,x_2)$ by $(b3)$. Now applay $(s1)$ on the vertices $\bar{u}, x_1, \bar{v}, x_2$, we get $\bar{v} \in R(x_1,x_2)$. Also $\bar{u}\in R(x_1,v)$, $x_2 \in R(\bar{u}, v)$ implies $x_2 \in R(x_1,v)$ by $(b2)$ and $\bar{v} \in R(x_1,x_2)$, $x_2 \in R(x_1,v)$ implies $\bar{v} \in R(x_1,v)$ by $(b2)$. By $(s1)$ on $\bar{u}, x_1, \bar{v}, v$, we get $v \in R(\bar{u}, \Bar{v})$, a contradiction. Therefore $\bar{u} \notin R(  x_1, v ) $, hence the claim.
Now we have $ x_1, \bar{v} \notin  R( \bar{u}, v )  $ and $ \bar{u} \notin R(  x_1, v ) $.
	Then $ R( \bar{u}, x_{1} ) \cap  R( \bar{u}, v) = \{   \bar{u} \} $,  $ R( \bar{u}, x_{1} ) \cap  R(x_{1}, v) = \{x_{1} \} $, $ R( \bar{u}, x_{1}) = \{ \bar{u} , x_{1} \} $, then we have to consider two cases\\
	\textbf{case 2.1}  
    	If $ R( \bar{u}, v) \cap   R(x_{1}, v) =  \{   v \}  $, then by $(ta)$ on $ \bar{u}, x_{1},v $, we get $ R( \bar{u}, v) = \{   \bar{u}, v \} $ and $ R(x_{1}, v) = \{   x_{1},v \}$. Now consider $\bar{u}, v, \bar{v}$, we have $ R(\bar{u}, v) = \{ \bar{u},v \}$, $R(v,\bar{v})= \{v,\bar{v}\}$ and $x_1 \notin R(\bar{u},v)$. Apply $(ta)$ on $\bar{u}, \bar{v}, v$, implies that $R(\bar{u}, \bar{v})=\{\bar{u}, \bar{v}\}$, contradiction to $x_1 \in R(\bar{u}, \bar{v})$.\\ 
	%Consider $u,v, \bar{v} $, also  $    x_{1}\notin R( u,v ) $, and by applying $(tc)$.	we get  $  R( u,v) = \{ u,v  \} $,  a contradiction.\\
	\textbf{case 2.2}  
	$ R( \bar{u}, v) \cap   R(x_{1}, v) \neq  \{  v \}  $. Then there exist $ x_{2} $ such that $ x_{2} \in  R( \bar{u}, v) \cap   R(x_{1}, v) $, and also by the lemma~\ref{s1s2} we may assume that $ R( \bar{u}, x_{2}) \cap   R(x_{1}, x_{2}) = \{  x_{2} \} $ and by $(ta)$ on $ x_{2},x_{1},  \bar{u}  $ we have $ R( \bar{u}, x_{2})  = \{  \bar{u}, x_{2}   \}  $ and  $ R( x_{1}, x_{2})  = \{  x_{1}, x_{2}   \}  $. Also $ x_{2} \in R( \bar{u}, v) \subseteq R( u, v)$.
	That is $ x_{2} \in R( u, v)  $, but  $ x_{2} \notin R( u, \bar{v})$.
	Now consider $  x_{1},x_{2}, \bar{v} $, by $(ta)$ we get a $ x_{3} $ such that $ x_{3} \notin R( \bar{u}, v)  $, but  $ x_{3} \in R( u, \bar{v})  $. Continuing like this we get  $  x_{4},x_{5},...,  x_{n} $ such that $x_n= v$ or $x_n= \bar{v}$. If $x_n= v$, then	$ x_n \in R( u, v)  $, but  $ x_n \notin R( u, \bar{v})  $ and if $x_n= \bar{v}$, then
	$x_n \notin R( u, v)  $, but  $ x_n \in R( u, \bar{v})  $. Therefore in all the cases we have $ R( u, \bar{v} ) \cap  R( v, \bar{v} ) = \{ \bar{v} \} , R( u, v ) \cap  R( v, \bar{v} ) = \{ v \} $ and 
	$ R( u, v ) \cap  R( u, \bar{v} ) = \{ u \} $ and 
	$  R( v,\bar{v}) = \{ v,\bar{v}  \} $,  and by $(ta)$ on  $ v,u, \bar{v}$,  we get  $  R( u,v) = \{ u,v  \} $, which is the final contradiction, since $ \bar{u}\in  R( u,v) $.
	Therefore $ \bar{u}  \in R( u,\bar{v} ) $. 
\end{proof}
The following straightforward Lemma for the connectedness of the underlying graph $G_R$ of a transit function $R$ is proved in \cite{mcjmhm-10}.

\begin{lemma}\label{connected}\cite{mcjmhm-10}
	If the transit function $R$ on a non-empty set $V$ satisfies axioms $(b1)$ and $(b2)$, then the underlying graph $G_R$ of $R$ is connected.
\end{lemma}
We have the following Theorem.
	
\begin{theorem}\label{b2}
 If $R$ is a transit function on $V$ satisfying the axioms  $(J0')$ and $(b3)$, then $R$ satisfies axiom $(b2)$ and $G_R$ is connected.   
\end{theorem}
	
\begin{proof}

Let $R$ satisfies axioms $(J0')$ and $(b3)$.  
To prove $R$ satisfies $(b2)$.   For  $u,x,y,v\in V$, let  $x\in R(u,v)$,  and $y\in R(u,x)$.  Since $R$ satisfies $(b3)$, we have   $x\in R(u,v), y\in R(u,x)\implies x\in R(y,v)$. In  axiom $(J0')$ the four elements are distinct, so the required minimum cardinality of the set $V $ is four and for $\mid V \mid \leq 3$, $(b2)$ holds trivially. We use induction on $\mid V \mid \geq 4$ and we explicitly prove cases where 	$\mid V \mid =4,5,6,7$. We also use the fact that axiom $(b3)$ implies $(b1)$.\\
\textbf{Case 1:} $\mid V \mid =4$.\\	
	Let $ V= \{u,x,y,v\} $. Here $ y\in R(u,x),$ $ x\in R(y,v)$. Since $x\in R(u,v)$, we have  $v\notin R(u,x)$ by $(b1)$.  Then $  R(u,x)\cap R(y,v) \subset \{u,x,y,v\} $  implies that $  y\in R(u,v) $  by  $(J0')$.\\
\textbf{Case 2:} $\mid V \mid =5$.\\	
Let $ V= \{u,x,w_1,y,v\} $. We have $ y\in R(u,x),$ $ x\in R(y,v)$ and if $  R(u,x)\cap R(y,v) \subset \{u,x,y,v\} $ implies that $  y\in R(u,v) $ by $(J0')$ and we are done. Assume  $w_1 \in  R(u,x)\cap R(y,v)  $ with $ w_1\notin \{u,x,y,v\} $. Now apply $(b3)$ we get the following.  $ x\in R(u,v), w_1\in R(u,x)   \implies  x\in R(w_1,v) $ , $ w_1\in R(y,v), x\in R(w_1,v)   \implies   w_1\in R(y,x) $  and $ y\in R(u,x), w_1\in R(y,x)  \implies  y\in R(u,w_1) $.\\
\textbf{Claim}  $x,v \notin R(u,w_1)$.
 Since  $ w_1\in R(u,x) $, $x \notin R(u,w_1) $ by $ (b1) $. Suppose $ v \in R(u,w_1) $, also we have $ x\in R(v,w_1) \implies v \in R(u,x) $  by  $(b3)$, a contradiction to $ (b1) $ as $ x \in R(u,v) $. That is  $x,v \notin R(u,w_1) $. Hence the claim.
 Now we have $ y\in R(u,w_1),  w_1\in R(y,v) $ and $  R(u,w_1) \cap R(y,v) \subset \{u,w_1,y,v\}  \implies y\in R(u,v) $  by  $(J0')$.\\
 \textbf{Case 3:} $\mid V \mid =6$.\\
Let $ V= \{u,x,w_1,w_2, y,v\} $. Here also $ y\in R(u,x),$ $ x\in R(y,v)$ and if $  R(u,x)\cap R(y,v) \subset \{u,x,y,v\} $ implies that $ y\in R(u,v) $. If $  R(u,x)\cap R(y,v) \not\subset \{u,x,y,v\} $. That is either $ w_1 $ or $ w_2 $ or both belongs to $ R(u,x)\cap R(y,v) $. Let us assume  $w_1 \in  R(u,x)\cap R(y,v)$ with $ w_1\notin \{u,x,y,v\} $. Then similar to  case $2$, we get either $y\in R(u,v) $, if $  R(u,w_1) \cap R(y,v) \subset \{u,w_1,y,v\} $ or $ w_2 \in R(u,w_1) \cap R(y,v) $ with $w_2 \notin \{u,x,w_1,y,v\} $. Then by $(b3)$, we get $w_1 \in R(w_2, x)$, since $w_1\in R(u,x), w_2\in R(u,w_1)$.\\
\textbf{Claim}  $u,y \notin R(w_1,v)$.
 Since  $ w_1\in R(y,v) $, $y \notin R(w_1,v) $ by $ (b1) $. Suppose $ u \in R(w_1,v) $ and we have $ w_1\in R(y,v) , u \in R(w_1,v) \implies  w_1 \in R(y,u) $ by $(b3)$, a contradiction to $ (b1) $ as $ y \in R(u,w_1) $. Hence the claim. Now $ x \in R(v,w_1), w_1 \in R(x,w_2) $ and $R(v,w_1) \cap R(x,w_2) \subset \{v,w_1,x,w_2\}$ implies $ x,w_1 \in R(v, w_2) $ by  $(J0')$.  Using $(b3)$ we have the following. $ w_2 \in R(y,v),x \in R( w_2,v)  \implies   w_2 \in R(y,x) $  and $y \in R(u,x) ,w_2 \in R(y,x)   \implies   y\in R(u,w_2) $.\\
 \textbf{Claim}  $x,v,w_1 \notin R(u,w_2) $.
 Since  $ w_2\in R(u,w_1) $ by $ (b1) $ we get $w_1 \notin R(u,w_2) $. Suppose $ x \in R(u,w_2) $ and we have $ w_1\in R(x,w_2)  \implies  x \in R(u,w_1) $ by $(b3)$, a contradiction to $ (b1) $ as $ w_1 \in R(u,x) $.  Suppose $ v \in R(u,w_2) $ and we have $ x\in R(v,w_2) $ then by  $(b3)$, $ v \in R(u,x) $, a contradiction to $ (b1) $ as $ x \in R(u,v) $. Hence the claim. Thus $R(u,w_2) \cap R(y,v) \subset \{u,w_2,y,v\} $. That is $ y\in R(u,w_2),  w_2\in R(y,v) $ and $  R(u,w_2) \cap R(y,v) \subset \{u,w_2,y,v\} $ implies that $y\in R(u,v) $  by  $(J0')$.\\
 \textbf{Case 4:} $\mid V \mid =7$.\\	
 Let $ V= \{u,x,w_1,w_2,w_3, y,v\} $. Then similar to  case $3$, we can prove the following.  Either $y\in R(u,v) $, or $ w_3 \in R(u,w_2) \cap R(y,v) $ with  $w_3\notin \{u,x,w_1,w_2,y,v\} $ and by  $(b3)$ we get $w_2 \in R(w_3, w_1)$, since $w_2\in R(u,w_1), w_3\in R(u,w_2)$. Also it is easy to get  $u,y,v \notin R(w_2,x) $. Then by $(J0)'$ on $x,w_1,w_2,w_3$, we get $ w_1,w_2 \in R(x, w_3) $. Using  $(b1)$ and $(b3)$ we can  easily  prove the Claims: $u,y, \notin R(w_2,v) $ and $x \notin R(w_3,w_1)$, since $ w_1,w_2 \in R(x, w_3) $. Then $ w_1 \in R(v,w_2), w_2 \in R(w_1,w_3) $ and $R(v,w_2) \cap R(w_1,w_3) \subset \{v,w_1,w_2,w_3\}$ implies $ w_1,w_2 \in R(v, w_3) $ by  $(J0')$. Then by continuous application of axiom $(b3)$, we get $y\in R(u,w_3)$ and we can prove the claim  $x,v,w_1,w_2 \notin R(u,w_3) $ easily as in  case $3$.  Then $ y\in R(u,w_3),  w_3\in R(y,v) $ and $  R(u,w_3) \cap R(y,v) \subset \{u,w_3,y,v\} $ implies that $y\in R(u,v) $  by axiom $(J0')$.\\
\textbf{Case $n$:} $\mid V \mid =n, n > 7$.\\	
 Let $ V= \{u,x,w_1,w_2,w_3,\ldots,w_n, y,v\} $. Like the above cases, we have either $y\in R(u,v)$ or $w_n \in R(u,w_{n-1}) \cap R(y,v)$ with  $w_n \notin \{u,x, w_1,w_2,\ldots,w_{n-1},y,v\}$. From the above cases we get the following: $w_i \in R(u,w_{i-1})$, $ w_{i-2},w_{i-1} \in R(w_{i-3}, w_i)$, $ w_{i-2},w_{i-1} \in R(v, w_i)$, $y\in R(u, w_i)$ and $w_{i-1} \in R(w_i, w_{i-2})$ for $4 \leq i \leq n-1$.  Using these we can easily prove the case $\mid V \mid = n$ like the above cases. Now $w_{n-1} \in R(u,w_{n-2})$, $ w_n \in R(u,w_{n-1}) \implies w_{n-1} \in R(w_n,w_{n-2})$ by  $(b3)$.
Also it is easy to get the claim $u,y,v,x, w_1,\ldots,w_{n-4} \notin R(w_{n-1},w_{n-3}) $. Then by $(J0)'$ on $w_{n-3},w_{n-2},w_{n-1},w_n$, we get $ w_{n-2},w_{n-1} \in R(w_{n-3}, w_n) $. Also we can   prove the Claim: $u,y \notin R(w_{n-1},v) $ and $x, w_1,\ldots,$ $ w_{n-3} \notin R(w_n,w_{n-2})$ using axiom $(b1)$ and $(b3)$ easily.  That is $ w_{n-2} \in  R(v,w_{n-1}),$ $ w_{n-1} \in R(w_n,w_{n-2}) $ and $R(v,w_{n-1}) \cap R(w_n,w_{n-2}) \subset \{v,w_{n-2},w_{n-1},w_n\}\implies  w_{n-2},w_{n-1} \in R(v, w_n) $ by  $(J0')$. Then by $(b3)$ we have the following.
 $w_n \in R(y,v), w_{n-1} \in R(w_n,v) \implies w_n \in R(y, w_{n-1})$ and  $y\in R(u,w_{n-1})$, $w_n \in R(y, w_{n-1}) \implies y \in R(u, w_n)$. We can prove the claim  $x,v,w_1,w_2,\ldots,w_{n-1} \notin R(u,w_n) $ easily as in the above cases. Then $ y\in R(u,w_n),  w_n\in R(y,v) $ and $  R(u,w_n) \cap R(y,v) \subset \{u,w_n,y,v\} $ implies that $y\in R(u,v) $  by axiom $(J0')$, which completes the proof that $R$ satisfies axiom $(b2)$. 
  Since for a transit function $R$, axiom $(b3)$ implies axiom $(b1)$ and by Lemma~\ref{connected}, it follows that $G_R$ is connected.
\end{proof}

\section{Interval function of diamond-weakly modular graphs}\label{interval_function}
In this section, we prove the main Theorem characterizing a diamond-weakly modular graph $G$ using the $(J0')$ axiom on the interval function $I_G$ of $G$, which in turn provide us an axiomatic characterization of the interval function $I_G$ of $G$.

\begin{theorem}\label{TDC}
 The interval function $I_G$ of a connected graph $G$ satisfies the axiom $(J0')$ if and only if $G$ is a diamond-weakly modular graph.
\end{theorem}
\begin{proof}
We prove the contra-positive of the statement of the theorem.
 That is, we prove that the interval function $I_G$ of a connected graph $G$ doesn't satisfy the axiom $(J0')$ if and only if $G$ is not a diamond-weakly modular graph. 
 First assume $G$ is not a diamond-weakly modular graph. Then the distance function $d$ doesn't satisfy either  $(TDC)$ or $(QC)$ or both on $G$.  Assume that $d$ doesn't satisfy the $(TDC)$. That is, for any three vertices $u, v, y$ with $1 = d(v, y) < d(u, v) = d(u, y)=k$ and there does not exists a common neighbor $z$ of $v,x,x'$ and $y$ such that $d(u, z) = d(u, v) - 1$ where $x$ is the neighbor of $y$ in the $u,y$-shortest path and $x'$ is the neighbor of $v$ in the $u,v$-shortest path. Then clearly, $x\in I(u,y)$, $y\in I(x,v)$ and $I(u,y)\cap I(x,v)\subset \{u,x,y,v\}$, but $x\notin I(u,v)$.  So $I$ do not satisfy axiom $(J0')$ on $G$. Now assume that $d$ doesn't satisfy the $(QC)$. That is, for any four vertices $u,v, w, y$ with $d(v, w) = d(w, y) = 1$ and $2 = d(v, y) \leq d(u, v) = d(u, y) = d(u, w) - 1$ and there does not exists a common neighbor $z$ of $v$ and $y$ such that $d(u, z) = d(u, v) - 1$. Let $x$ be the neighbor of $y$ in the $u,y$-shortest path. Then clearly $x\in I(u,y),$ $y\in I(x,v)$ and $I(u,y)\cap I(x,v)\subset \{u,x,y,v\}$, but  $x\notin I(u,v)$. So $I$ does not satisfy axiom $(J0')$ on $G$.
 
 To prove the converse part: Assume that the interval function $I$ of a connected graph $G$ doesn't satisfy the axiom $(J0')$. That is, there exists vertices $u,x,y,v$ in $G$ such that $x\in I(u,y)$, $y\in I(x,v)$, $I(u,y)\cap I(x,v) \subset \{u,x,y,v\}$ and $x\notin I(u,v)$. That is, there exist a $u,y$-shortest path containing $x$ and a $x,v$-shortest path containing $y$ such that $x$ doesn't belong to a $u,v$-shortest path in $G$. From this, we infer that $x$ and $y$ are adjacent and that $d(u,y),d(x,v), d( u,v) \geq 2$. Without loss of generality, we can take the vertices $u,x,y,v$ such that $y$ is at a maximum distance  say $m$ from $u$ and let $d(y,v)=n$. 

   Consider the vertex $y_1 \in I(y,v)$ which is adjacent to $y$. We may choose $y_1$ such that the distance  $d(u,y_1)$ is minimum. Then the possibilities of  $d(u,y_1)$ are $m-1$, $m$ or $m+1$. If $d(u,y_1) = m-1$, then $y_1 \in I(u,y) \cap I(x,v)$, which is a contradiction to the assumption that $u,x,y,v$ does not satisfy axiom $(J0')$, so that $d(u,y_1) = m-1$ is not possible.

When $d(u,y_1) = m$: consider the vertices $u$, $x$, $y$, $y_1$. Clearly, these vertices do not satisfy the axiom $(J0')$. Let $y^\prime \in I(u,y_1)$ which is adjacent to $y_1$. Then, $d(u,y)=d(u,y_1) = m$ and $d(y,y_1) = 1$. If possible, let the vertices $u$, $y$, $y_1$ satisfies $(TDC)$. Then there exists a vertex $z$ adjacent to both $y$ and $y_1$ with $d(u,z)=m-1$ and $z$ is adjacent to both $x$ and  $y^\prime$. Then, $z\in I(u,y)\cap I(x,v)$ and so $I(u,y) \cap I(x,v)\not\subset \{u,x,y,v\}$ and hence the vertices $u$, $x$, $y$, $v$ satisfies the axiom $(J0')$. Hence the vertices $u$,  $y$, $y_1$ will not satisfy $(TDC)$  and so $G$ is not diamond-weakly modular graph.

When $d(u,y_1) = m+1$: consider the vertices $u$, $y$, $y_1$, $v$. Clearly, $y\in I(u,y_1)$ and $y_1 \in I(y,v)$. If $I(u,y_1) \cap I(y,v) \subset \{u,y_1,y,v \}$, then either $y \in I(u,v)$ or $y\notin I(u,v)$. But both of these will yield a contradiction. For, if $y \in I(u,v)$ then $x\in I(u,v)$ (by axiom $(b2)$); if $y \notin I(u,v)$, then we get four vertices $u$, $y$, $y_1$, $v$ which does not satisfy $(J0')$ and $d(u,y_1)>m$, a contradiction to the maximality of $d(u,y)$. So this implies that $I(u,y_1) \cap I(y,v) \not\subset \{u,y_1,y,v \}$. Then, there exist a $w\neq u,y_1,y,v$ such that $w\in I(u,y_1) \cap I(y,v)$. Let $P$ be the shortest $uy_1$-path containing $w$ and   $Q$ be the shortest $y,v$-path containing $w$. Also let  $t$  be the length of shortest $u,w$- subpath of $P$ and $\ell $ be the length of shortest $w,y$- subpath of $Q$ .\\
\textbf{Claim}: $\ell+t=m+2$.  We have $d(u,y)=m$ and $d(u,y_1)=m+1$. So possibilities of $\ell+t $ are $m,m+1$ and $m+2$. Let $w_1$ be the neighbor of $y $ in the shortest $w,y$- subpath of $Q$. Also $w_1 \in I(y,v)$. So $d(u,w_1)=m-1$ and $d(u,w_1)=m$ are not possible, since it is a contradiction to the choice of $y_1$ that $d(u,y_1)$ minimum.
%If $\ell+t=m $, then $d(u,w_1)=m-1$ and hence $w_1 \in I(u,y) \cap I(x,v)$. So $\ell+t\neq m $. If $\ell+t=m+1 $. We have $d(u,w_1)=m-1$ is not possible Since   $w_1 \notin I(u,y) \cap I(x,v)$. That is $d(u,w_1)=m$, which is a contradiction to the choice of $y_1$ that $d(u,y_1)$ minimum.So $\ell+t\neq m+1 $.
Therefore $\ell+t=m+2$. Hence the claim. That is  $ w \rightarrow y_1\rightarrow y$ is a shortest $w,y$- path.  So can replace $w_1$ with $y_1$.  Consider the  vertices $u$, $y$, $y_1$ $y^{\prime}$, it is clear that $d(u,y)=d(u,y^{\prime}) =m$ and $d(y,y^{\prime})=2$. If $d$ satisfies $(QC)$, then there exist a $z$ which is adjacent to $y$ and $y^{\prime}$ with $d(u,z)=m-1$. Then $z\in I(y,w)$ and since $ I(y,w)\subseteq I(y,v)$  implies that $z\in I(y,v) $. Also $z \in  I(u,y)$ and hence $z \in  I(u,y)\cap I(x,v)$. So the vertices $u$, $x$, $y$, 
satisfies $(J0')$, contradiction. Hence  $d$  does not satisfy $(QC)$ on $u,y,y^{\prime}$ and hence $G$ is not diamond-weakly modular.

That is, we have proved that if $I$ doesn't satisfy axiom $(J0')$, then $G$  doesn't satisfy $(TDC)$ or  $(QC)$ and hence $G$ is not a diamond-weakly modular graph.

\end{proof}

From  Proposition~\ref{Ws1}, Theorem~\ref{interval}, Theorem~\ref{b2}, Theorem~\ref{TDC} and  we have the following theorem.
\begin{theorem}\label{WMD}
Let $R$ be a function from $V\times V$ to $2^V$, where $V$ is a non empty set. Then $R$  satisfies the axioms $ (t1), (t2), (t3), (b3), (J0')$ and $(ta)$   if and only if  $G_R$ is a diamond-weakly modular graph and $R$ coincides with the interval function $I_{G_R}$ of $G_R$.
\end{theorem}
The examples below establishes the independence of the axioms $(t1)$, $(t2)$, $(t3)$, $(b3)$, $(J0')$ and $(ta)$.
\begin{example}[$(t2)$, $(t3)$, $(b3)$, $(J0')$ and $(ta)$ but not $(t1)$]$~$\label{ex1}\\
	Let $V=\{u,v,x,y\}$ and define a transit function $R$ on $V$ as $R(u,v)=R(v,u) =\{u\},$ $R(u,x)=\{u,x\}$ $R(u,y)=\{u,x,y\}, R(v,x)=\{v,x\},$ $R(v,y)=\{v,x,y\}$  $R(x,y)=\{x,y\}$, $R(a,a)= \{a\}$ and $R(a,b)=R(b,a)$ for all $a,b\in V$.  We can see that $R$ satisfies $(t2)$, $(t3)$, $(b3)$, $(J0')$ and $(ta)$.  But $R$ does not satisfy axiom $(t1)$, since $v\notin R(u,v)$ and $v\notin R(v,u)$.
	\end{example}
	\begin{example}[$(t1)$, $(t3)$, $(b3)$, $(J0')$ and $(ta)$ but not $(t2)$]$~$\label{ex2}\\
	Let $V=\{u,v,x,y\}$ and define a transit function $R$ on $V$ as follows: $R(u,v)=\{u,v\}=R(v,u), R(u,x)=\{u,v,x\}$, $R(x,u)=\{x,u\}$  $R(u,y)=\{u,v,x,y\}=R(y,u), R(v,x)=\{v,x\}= R(x,v),$ $R(v,y)=\{v,x,y\}$, $R(y,v)=\{y,v\},$ $R(x,y)=\{x,y\}=R(y,x)$, $R(a,a)= \{a\}$ for all $a\in V$.  We can see that $R$ satisfies $(t1)$, $(t3)$, $(b3)$, $(J0')$ and $(ta)$.  But $R(u,v) \neq R(v,u)$.  Therefore $R$ does not satisfy the $(t2)$ axiom.
\end{example}
	\begin{example}[$(t1)$, $(t2)$, $(b3)$, $(J0')$ and $(ta)$ but not $(t3)$]$~$\label{ex3}\\
	Let $V=\{u,v,x,y\}$ and define a transit function $R$ on $V$ as follows: $R(u,u)=\{u,v\}, R(u,v)=\{u,v\}, R(u,x)=\{u,x\}$,  $R(u,y)=\{u,y\}, R(v,x)=\{v,u,x\},$ $R(v,y)=\{v,u,y\}$  $R(x,y)=\{x,y\}$, $R(a,a)= \{a\}$ and $R(a,b)=R(b,a)$ for all $a,b \in V$.  We can see that $R$ satisfies $(t1)$, $(t2)$, $(b3)$, $(J0')$ and $(ta)$.  But $v \in R(u,u) $.  Therefore $R$ does not satisfy the $(t3)$ axiom.
\end{example}
	\begin{example}[$(t1)$, $(t2)$, $(t3)$, $ (b3)$, $(ta)$ but not $(J0')$]$~$\\
	Let $V=\{u,v,w,x,y,z\}$ and define a transit function $R$ on $V$ as follows: $R(u,v)=\{u,w,v,z\}, R(x,v)=\{x,w,y,v\}$, $R(x,z)=\{x,y,u,z\}$  $R(w,z)=\{w,u,v,z\}, R(w,y)=\{w,x,v,y\},$ and $R(a,a)= \{a\}$ and $R(a,b)=R(b,a)$ for all $a,b \in V$.    We can see that $R$ satisfies $(t1)$, $(t2)$, $(t3)$, $ (b3)$, $ (ta)$. But $x\in R(u,y), y\in R(x,v), R(u,y)\cap R(x,v) \subset \{u,x,y,v\}$ and $ x\notin R(u,v)$  Therefore $R$ does not satisfy the $(J0')$ axiom.
\end{example}

\begin{example}[$(t1)$, $(t2)$, $(t3)$, $ (b3)$, $(J0')$ but not $(ta)$]$~$\\
	Let $V=\{u,v,w,x,y\}$ and define a transit function $R$ on $V$ as follows: $R(u,w)=\{u,x,y,w\}, R(v,x)=\{v,y,w,x\}$ and $R(a,a)= \{a\}$ and $R(a,b)=R(b,a)$ for all $a,b \in V$.  We can see that $R$ satisfies $(t1)$, $(t2)$, $(t3)$, $ (b3)$. Also $w\in R(v,x), x\in R(w,u)$ and $ R(v,x)\cap R(w,u)= \{y,x,w\}\not\subset \{v,x,w,u\}$. That is $R$ satisfy $(J0')$. But  $ R( u,v) \cap R( u,w) =\{  u \} , R( u,v) \cap R( v,w) =\{  v \}, R( u,w) \cap R( v,w) =\{  w \} $ and $ R( u,v) = \{ u,v  \}  ,  R( v,w) = \{ v,w  \}   $ and $  R( u, w) = \{ u,x,y, w  \}$. Therefore $R$ does not satisfy the $(ta)$ axiom.
\end{example}
\begin{example}[$(t1)$, $(t2)$, $(t3)$, $(ta)$,  $(J0')$ but not $(b3)$]$~$\label{exb3}\\
	Let $V=\{u,v,w,x,y\}$ and define a transit function $R$ on $V$ as follows: $R(u,y)=\{u,x,y,w\}, R(x,v)=\{x,y,v,w\}$,   $R(u,v)=V$ and $R(a,a)= \{a\}$ and $R(a,b)=R(b,a)$ for all $a,b \in V$.   We can see that $R$ satisfies $(t1)$, $(t2)$, $(t3)$, $ (J0')$ and $(ta)$.  But $y\in R(u,v), w\in R(u,y)$, and $y\notin R(w,v)$. Therefore $R$ does not satisfy the $(b3)$ axiom.
\end{example}

\section{Interval function of bridged graphs and  weakly bridged graphs}\label{Bridged}

%It may be noted that bridged graphs form a strict subclass of weakly modular graphs. %In \cite{chepoi}, Choepoi et al. characterize weakly modular graph using  $(tc)$ axiom. 
In this section, we characterize the interval function of bridged graphs using axioms $(J0')$ and  $(br)$ and weakly bridged graphs by $(J0')$ and  $(br')$ on an arbitrary transit function.
First, we prove the following Theorem characterizing the interval function  $I$ of the bridged graph.

\begin{theorem}\label{bridged}
	Let $G$ be a graph. The interval function $I_G$ satisfies the axioms $(J0')$ and $(br)$ if and only if $G$ is a bridged graph.
\end{theorem}
\begin{proof}
First assume that $I$ satisfy axioms $(J0')$ and $(br)$ on $G$. We have to show that $G$ is a bridged graph. By Theorem~\ref{TDC}, the interval function $I$ on a graph $G$ satisfies $(J0')$ if and only if $G$ is a diamond-weakly modular graph. That is, if $I$ satisfy axiom $(J0')$,  then $G$ is a diamond-weakly modular graph which is a subclass of weakly modular graphs.  According to the result by Chepoi et al. \cite{Ch_metric}, bridged graphs are exactly weakly modular graphs that do not contain induced $C_4$ and $C_5$. Now it is enough to show that $I$ satisfy $(br)$, then $G$ is $(C_4,C_5)-$ free.   If $G$ contains an induced   $C_5$ with consecutive vertices as  $u,x,y,v,z$, it is clear that  $I$ does not satisfy $(br)$ on  $C_5$, since $I(x,y)=\{x,y\}, I(x,u)=\{x,u\}, I(v,y)=\{v,y\} $ and $z\in I(u,v) $ and $ I(x,z)\neq \{x,z\}$ and $I(y,z)\neq \{y,z\}$. Similarly, if  $G$ contains an induced   $C_4$ with consecutive vertices as $u,x=y,v,z$,  it is clear that  $I$ does not satisfy $(br)$ on  $C_4$. 

Conversely, suppose that $G$ is a bridged graph. We have to show that  $I$  satisfy axioms $(J0')$ and $(br)$. Since bridged graphs are diamond-weakly modular graphs,   $I$  satisfy axiom $(J0')$ on $G$. It remains  to show that $I$  satisfy axiom $(br)$. Suppose not. That is $I(x,y)=\{x,y\}, I(x,u)=\{x,u\}, I(v,y)=\{v,y\} $ and $z\in I(u,v) $ and $ I(x,z)\neq \{x,z\}$ and $I(y,z)\neq \{y,z\}$. Let $P:uxyv$ be the $u,v$ path containing $x$ and $y$. Since $z\in I(u,v) $, $ I(x,z)\neq \{x,z\}$ and $I(y,z)\neq \{y,z\}$, there exist a $u,v$- shortest path containing $z$ which is different from the path $P$.
Then possibility of $d(u,v)$ are $3$ and $2$. When $d(u,v)=3$, then $z$ is either adjacent to $u$ or $v$. Without loss of generality, we may assume $z$ is adjacent to $u$. Let $z\prime$ be the neighbor of $z$  in the $u,v$- shortest path. Then $u,x,y,v,z\prime, z$ will form a cycle of length $6$.  If $z\prime y \in E(G)$ and $z\prime x \notin E(G)$, the vertices $u,x,y,z\prime ,z$ will induce a cycle of length $5$.   If  both $z\prime y, z\prime x \in E(G)$, then the vertices $u,x,z\prime,z$ will induce a $C_4$. When $d(u,v)=2$, then $z$ is  adjacent to  both $u$ and $v$. Then the vertices $u,x,y,z,v$ will induce a cycle of length $5$. If $uy\in E(G)$, the vertices $u,y,z,v$ will induce a cycle of length $4$ and if $xv\in E(G)$, the vertices $u,x,v,z$ will induce a cycle of length $4$. When $x=y$, the vertices $u,x=y,z,v$ will induce a $4$-cycle. %In this case we rename the vertices as $x=y$ and $u$ and $v$ be the adjacent vertices on both side of $x$ and the remaining one be $z$. the vertices $u,x=y,v,z$ will induces a $C_4$ or $W_4$.
In all cases, we get a contradiction to the assumption that $G$ is a bridged graph.
\end{proof}
From  Proposition~\ref{Ws1}, Theorem~\ref{interval}, Theorem~\ref{b2},  and Theorem~\ref{bridged}, we have the following theorem.
\begin{theorem}\label{bri}
Let $R$ be a function from $V\times V$ to $2^V$, where $V$ is a non empty set. Then $R$  satisfies the axioms $(t1)$, $(t2)$, $(t3)$, $(J0')$, $(b3)$, $(ta) $ and $(br)$   if and only if  $G_R$ is a bridged graph and $R$ coincides with the interval function $I_{G_R}$.
\end{theorem}
The examples below establish the independence of the axioms $(t1)$, $(t2)$, $(t3)$, $(b3)$, $(br)$, $(J0')$ and $(ta)$.
 The function defined in Example~\ref{ex1} satisfy axioms  $(t2)$, $(t3)$, $(b3)$, $(br)$, $(J0')$ and $(ta)$ but does not satisfy axiom $(t1)$. The function defined in Example~\ref{ex2} satisfy axioms  $(t1)$, $(t3)$, $(b3)$, $(br)$, $(J0')$ and $(ta)$ but does not satisfy axiom $(t2)$. The function defined in Example~\ref{ex3} satisfy axioms  $(t1)$, $(t2)$, $(b3)$, $(br)$, $(J0')$ and $(ta)$ but does not satisfy axiom $(t3)$. 
\begin{example}[$(t1)$, $(t2)$, $(t3)$, $(ta)$, $(br)$, $(b3)$ but not $(J0')$]$~$\\\label{e1}
    Let $V=V(C_8)$ and $R=I$ on $C_8$. The consecutive vertices of $C_8$ be $v_1, v_2,..., v_8$. Then $R$ satisfies axioms $(t1)$, $(t2)$, $(t3)$, $(br)$, $(ta)$ and $(b3)$. Assume $u=v_1, x=v_3, y=v_4$ and $v= v_6$. Then $x\in R(u,y), y\in R(x,v), R(u,y)\cap R(x,v) \subset \{u,x,y,v\}$ and $ x\notin R(u,v)$.  Therefore $R$ does not satisfy the $(J0')$ axiom.
\end{example}
\begin{example}[$(t1)$, $(t2)$, $(t3)$, $(br)$, $(b3)$, $(J0')$ but not $(ta)$]$~$\\\label{e2}
	Let $V=\{u,v,w,x,y\}$ and define a transit function $R$ on $V$ as follows: $R(u,w)=\{u,y,w\}, R(v,x)=\{v,y,w,x\}$, $R(u,x)=\{u,y,x\}$ and $R(a,a)= \{a\}$ and $R(a,b)=R(b,a)$ for all $a,b \in V$.  We can see that $R$ satisfies $(t1)$, $(t2)$, $(t3)$, $ (b3)$, $(br)$ and $(J0')$. But  $ R( u,v) \cap R( u,w) =\{  u \} , R( u,v) \cap R( v,w) =\{  v \}, R( u,w) \cap R( v,w) =\{  w \} $ and $ R( u,v) = \{ u,v  \}  ,  R( v,w) = \{ v,w  \}   $ and $  R( u, w) = \{ u,y, w  \}$. Therefore $R$ does not satisfy the $(ta)$ axiom.
\end{example}
\begin{example}[$(t1)$, $(t2)$, $(t3)$, $(ta)$, $(br)$, $(J0')$ but not $(b3)$]$~$\label{e3}\\
	Let $V=\{u,v,w,x,y\}$ and define a transit function $R$ on $V$ as follows: $R(u,y)=\{u,x,y,w\}, R(x,v)=\{x,y,v,w\}$,   $R(u,v)=V$ and $R(a,a)= \{a\}$ and $R(a,b)=R(b,a)$ for all $a,b \in V$.   We can see that $R$ satisfies $(t1)$, $(t2)$, $(t3)$, $ (J0')$, $(br)$ and $(ta)$.  But $y\in R(u,v), w\in R(u,y)$, and $y\notin R(w,v)$.  Therefore $R$ does not satisfy the $(b3)$ axiom.
\end{example}
\begin{example}[$(t1)$, $(t2)$, $(t3)$, $(ta)$, $(b3)$, $(J0')$ but not $(br)$]$~$\label{e4}\\
	Let $V=\{u,v,x,y,z,w\}$ and define a transit function $R$ on $V$ as follows: $R(u,y)=\{u,x,z,y\}$,$R(u,v)=\{u,v,z,w\} R(x,v)=\{x,y,z,v\}$, $R(x,w)=\{x,w,z,u\}$,$R(w,y)=\{w,y,z,v\}$   and $R(a,a)= \{a\}$ and $R(a,b)=R(b,a)$ for all $a,b \in V$.   We can see that $R$ satisfies $(t1)$, $(t2)$, $(t3)$, $ (J0')$, $(ta)$ and $(b3)$.  $R$ does not satisfy the $(br)$, since $I(x,y)=\{x,y\}, I(x,u)=\{x,u\}, I(v,y)=\{v,y\} $, $z\in I(u,v) $ and $ I(x,z)\neq \{x,z\}$ and $I(y,z)\neq \{y,z\}$.
\end{example}
%\section{Interval function of hereditary weakly modular graph}\label{hereditary}
%A graph $G$ is called hereditary weakly modular if every isometric subgraph of $G$ is weakly modular. Equivalently  a graph $G$ is hereditary weakly modular if and only if it does not contain induced houses and isometric $C_n$ with $n \geq 5$.  In this section, we characterize hereditary weakly modular graph by using $(j0')$ and  $(hr)$.
 For the characterization of the interval function of weakly bridged graphs, we need the Lemma~\ref{wbrid}. From the definition of  axioms $(br)$ and $(br')$, we can see that axiom $(br')$ is easily followed from axiom $(br)$ by taking $x=y$.
 
   The axioms $(br')$ and $(J0')$ characterize  the interval function of weakly bridged graphs. We have the following theorems in which the proof is similar to  that of theorems, Theorem~\ref{bridged} and ~\ref{bri}. 
\begin{theorem}\label{wbridged}
	Let $G$ be a graph. The interval function $I_G$ satisfies the axioms $(J0')$ and $(br')$ if and only if $G$ is a weakly bridged graph.
\end{theorem}  
  \begin{theorem}\label{wbri}
Let $R$ be a function from $V\times V$ to $2^V$, where $V$ is a non empty set. Then $R$  satisfies the axioms $(t1)$, $(t2)$, $(t3)$, $(J0')$, $(b3)$, $(ta) $ and $(br')$   if and only if  $G_R$ is a weakly bridged graph and $R$ coincides with the interval function $I_{G_R}$.
\end{theorem}
The independence of the axioms $(t1)$, $(t2)$, $(t3)$, $(b3)$, $(br')$, $(J0')$ and $(ta)$ will easily follow from Example~\ref{ex1}, Example~\ref{ex2}, Example~\ref{ex3}, Example~\ref{e1}, Example~\ref{e2}, and Example~\ref{e3} except the following one.
\begin{example}[$(t1)$, $(t2)$, $(t3)$, $(ta)$, $(b3)$, $(J0')$ but not $(br')$]$~$%\label{e4}\\
	Let $V=\{u,v,x,z\}$ and define a transit function $R$ on $V$ as follows: $R(u,v)= R(x,z)=V$    and $R(a,a)= \{a\}$ and $R(a,b)=R(b,a)$ for all $a,b \in V$.   We can see that $R$ satisfies $(t1)$, $(t2)$, $(t3)$, $ (J0')$, $(ta)$ and $(b3)$.  $R$ does not satisfy the $(br')$, since $ I(x,u)=\{x,u\}, I(v,x)=\{v,x\} $, $z\in I(u,v) $ and $ I(x,z)\neq \{x,z\}$.
\end{example}

\noindent \textbf{Concluding Remarks:}
In this paper we introduced a subclass of weakly modular graphs which satisfy a stronger version (diamond condition) of the triangle condition of the weakly modular graphs and name it as  the diamond-weakly modular graphs. We have proved that the class of diamond-weakly modular graphs contain the family of bridged graphs and  weakly bridged graphs. Further we have characterized this class of graphs in terms of the axioms of its interval function. Nevertheless, a purely graph theoretic characterization for diamond-weakly modular graphs remains as an open problem.

As a byproduct, we have obtained that the diamond-weakly modular graphs, bridged graphs and weakly bridged graphs are definable in FOLB. 

From the Mulder- Nebesk\'{y} Theorem (Theorem~\ref{interval}), the following definition of \emph{graphic interval structure} is derived by Chalopine at al. in \cite{Chalopin -2021}.  We quote the following definitions and conventions regarding the FOLB definability from \cite{Chalopin -2021}.

A \emph{graphic interval structure} is a $\sigma$-structure $(V,B)$
where $V$ is a finite set and $B$ is a ternary predicate ( a ternary relation) on $V$ satisfying the following axioms:

\begin{enumerate}
 \item [{(IB}1)] $\forall u \forall v  B(u,u,v) $
 \item [{(IB}2)] $\forall u \forall v \forall x  B(u,x,v) \implies B(v,x,u)$
\item [{(IB}3)]$\forall u \forall x B(u,x,u) \implies x=u $
\item [{(IB}4)] $\forall u \forall v \forall w \forall x  B(u,w,v) \wedge B(u,x,w) \implies B(u,x,v)$
\item [{(IB}5)]$\forall u \forall v \forall w \forall x B(u,v,x) \wedge B(u,w,x) \wedge B(u,v,w) \implies B(v,w,x)$
\item [{(IB}6)]$\forall u \forall u' \forall v \forall v' B(u,u',u') \wedge B(v,v',v') \wedge B(u',u,v') \wedge B(u,u',v) \wedge B(u,v',v) \implies B(u',v,v')$
\item [{(IB}7)] $\forall u \forall u' \forall v \forall v' B(u,u',u') \wedge B(v,v',v') \wedge B(u,u',v) \wedge \neg B(u,v',v) \wedge \neg B(u',v,v') \implies B(u,u',v')$.
\end{enumerate}

In the class $\mathcal{C}$ of $\sigma$-structures $(V,B)$
satisfying axioms (IB1)--(IB7), a query $Q$ on $\mathcal{C}$ is definable in first order logic with betweeness (\emph{FOLB-definable}), if it can be defined by a first order formula $F$ over $(V,B)$.

From the Theorems~ \ref{TDC}, \ref{bridged} and \ref{wbridged}, we can easily derive the following FOLB queries for diamond-weakly modular graphs, bridged graphs and weakly bridged graphs.

\emph{diamond-weakly modular} $\equiv$ $  \forall u,x,y,v$, $B(u,x,y)\wedge B(x,y,v) \wedge [(B(u,z,y) \wedge B(x,z,v) \implies (z=u \vee z=x \vee z=y) \wedge z \neq v) \vee ((B(u,z,y) \wedge B(x,z,v) \implies (z=u \vee z=x \vee z=v)\wedge z \neq y) \vee ((B(u,z,y) \wedge B(x,z,v) \implies (z=u \vee z=v \vee z=y) \wedge z \neq x )\vee ((B(u,z,y) \wedge B(x,z,v) \implies (z=v \vee z=x \vee z=y) \wedge z \neq u)] \implies  B(u,x,v)$.

\emph{bridged} $\equiv$ \emph{diamond-weakly modular}  $\wedge \forall u, x, y, v, z, w[(( B(x,w,y) \implies w=x \vee w=y )\wedge  ( B(x,w,u) \implies w=x  \vee w=u )\wedge ( B(v,w,y) \implies w=v  \vee w=y ) \wedge B(u,z,v))\implies (( B(x,w,z) \implies w=x  \vee w=z )\vee ( B(z,w,y) \implies w=z  \vee w=y ))]$

%\emph{hereditary weakly modular graph} $\equiv$   \emph{diamond-weakly modular}    $\wedge \forall u, x, y, v, z, w$ \\ $[(( B(x,w,y) \implies w=x \vee w=y )\wedge  ( B(x,w,u) \implies w=x \vee w=u )\wedge ( B(v,w,y) \implies w=v \vee w=y ) \wedge B(u,z,v) \wedge \neg B(u,x,v) ))\implies (( B(x,w,z) \implies w=x \vee w=z )\vee ( B(z,w,y) \implies w=z \vee w=y ))]$

\emph{weakly bridged} $\equiv$ \emph{diamond-weakly modular} $\wedge  \forall u, x, v, z, w $ $[(( B(u,w,x) \implies w=u \vee w=x )\wedge ( B(x,w,v) \implies w=x \vee w=v )\wedge B(u,z,v)) \implies (( B(x,w,z) \implies w=x \vee w=z ))] $

\end{document}